\documentclass[12pt,a4paper]{amsart}
\usepackage{amsmath,amsfonts,amssymb,amsthm,mathrsfs}
\usepackage{color}

\newtheorem{theorem}{Theorem}[section]
\newtheorem{definition}[theorem]{Definition}

\newtheorem{corollary}[theorem]{Corollary}
\newtheorem{lemma}[theorem]{Lemma}

\theoremstyle{remark}
\newtheorem{rem}[theorem]{Remark}

\usepackage{delarray}
\usepackage{enumerate}

\numberwithin{equation}{section}
\numberwithin{theorem}{section}
\numberwithin{figure}{section}

\renewcommand{\P}{\mathbb{P}}
\newcommand{\X }{{\mathcal{X}_\sigma}}
\newcommand{\ud }{\,{\rm d}}
\newcommand{\rf}[1]{\eqref{#1}}

\newcommand{\bbfR}{{\mathbb R}}

\newcommand{\bbfN}{{\mathbb N}}
\newcommand{\vf}{{\varphi}}
\newcommand{\ve}{{\varepsilon }}
\newcommand{\se}{{\text{\rm{e}}}}

\newcommand{\bfo}{{\mathfrak b}}

\newcommand{\R}{{\mathbb R}}
\date{\rule{0pt}{15pt}\today}
\begin{document}

\title[Navier-Stokes system]{$L^2$-asymptotic stability of mild solutions \\
to Navier-Stokes system in $\R^3$
}

\author [G.~Karch]{Grzegorz Karch}
\address[G.~Karch]{Instytut Matematyczny, Uniwersytet Wroc\l awski, pl. Grunwaldzki 2/4, 50-384 Wroc\l aw, Poland}
\email{grzegorz.karch@math.uni.wroc.pl}

\author[D.~Pilarczyk]{Dominika Pilarczyk}
\address[D.~Pilarczyk]{Instytut Matematyczny, Uniwersytet Wroc\l awski, pl. Grunwaldzki 2/4, 50-384 Wroc\l aw, Poland}
\email{dominika.pilarczyk@math.uni.wroc.pl}

\author[M.E.~Schonbek] {Maria E. Schonbek}
\address[M.E.~Schonbek]{Department of Mathematics, UC Santa Cruz, CA 95064, USA}
\email{schonbek@ucsc.edu}

\bigskip

 %%%%%%%%%%%%%%%%%%%%%%%%%%%%%%%%%%%%%%%%%%%%%%%%%%%%%%%%%%%%%%%%

%%%%%%%%%%%%%%%%%%%%%%%%%%%%%%%%%%%%%%%%%%%%%%%%%%%%%%%%%%%%%%%%%%%%%%%%%%%%%%%%%%%%%%%%%%%%%%%%%%%%%%%%%%%%%%%%%%%%%%%%%%%%%%%%
\begin{abstract}
We consider global-in-time small mild solutions of the initial value problem to the incompressible Navier-Stokes equations
 in $\bbfR^3$. For such solutions,  an asymptotic stability is established 
under arbitrarily large initial $L^2$-perturbations.
\date{\today}
\\
\\
{\small \it Mathematics Subject Classification (2000):}  76D07; 76D05; 35Q30; 35B40. \\
{\small \it Keywords:} {\small Navier--Stokes equations; mild solutions; weak solutions; asymptotic stability of solutions, Fourier splitting method}

\end{abstract}

\maketitle

\section{Introduction}
\setcounter{equation}{0}

The classical theory of viscous incompressible fluid flow is governed by the celebrated Navier-Stokes equations
\begin{align}
\label{NS eq}
         u_t-\Delta u+ (u \cdot \nabla ) u +\nabla p&=F,\quad  (x,t) \in \bbfR^3\times (0,\infty ),\\
\label{NS div}         
       {\rm div }\ u&=0, \\
\label{NS in d}       
       u(x,0)&=u_0 (x),
\end{align}  
where $u= \big(u_1 (x,t) ,u_2 (x,t) ,u_3 (x,t)\big)$ is the velocity of the fluid and $p = p(x,t)$    the scalar pressure. The functions  $u_0 = u_0 (x)$ and $F = F(x,t)$ denote a given initial velocity and an external force. A large area of modern research is devoted to deducing different  qualitative properties of solutions for the incompressible Navier-Stokes  equations. The work on this subject is too broad to attempt to give a complete list of references. We will   limit ourselves to  discussions directly connected  to issues in this paper, specifically,  questions on  stability of solutions in the whole three dimensional space.

There are two main approaches for the construction of solutions to the initial value problem \rf{NS eq}--\rf{NS in d}. In the pioneering  paper by Leray \cite{L}, weak solutions to \rf{NS eq}--\rf{NS in d} are obtained for all divergence free initial data $u_0 \in L^2 (\bbfR^3)^3$ and $F=0$. These solutions satisfy equations \rf{NS eq}--\rf{NS div} in the distributional sense and  fulfill a suitable energy inequality. Fundamental questions on  regularity and  uniqueness of the weak solutions to the 3D Navier-Stokes equations remain open, see {\it e.g.} the books \cite{T, Lemarie} and the review article \cite{Cannone}
for an additional background and references. 

The second approach leads to mild solutions. These solutions are given by an integral formulation using the Duhamel principle and  are obtained  by means of the Banach contraction principle. Specifically, mild solutions  are known to exist for large  initial conditions  on a finite time interval. For sufficiently small data, in appropriate scale-invariant spaces, the  corresponding mild solutions are global-in-time and their dependence on data is regular. We refer the reader to ~\cite{Cann-Book,Cannone, Lemarie} as well as to Section \ref{hardy_ineq} of our paper for a review of the theory on mild solutions 
to problem  \rf{NS eq}--\rf{NS in d}.

The goal of this work is to describe a  link between these two approaches. Our  result can be summarized as follows.
 Assume that $V = V(x,t)$ is a global-in-time mild solution of  \rf{NS eq}--\rf{NS in d}, small in  some scale-invariant space. We show that 
problem  \rf{NS eq}--\rf{NS div} has a global-in-time weak solution in the sense of Leray corresponding to
  the initial datum $V(x, 0)$ perturbed by an arbitrarily large divergence free $L^2$-vector field,
Moreover, this weak solution 
converges in the energy norm 
as $t\rightarrow \infty $ to the mild solution $V = V (x, t)$.
In other words, we show that a sufficiently small mild solution $V = V (x, t)$ of  problem \rf{NS eq}--\rf{NS div} is, in some sense, an asymptotically  stable weak solution  of this problem under all divergence free initial perturbations from $L^2 (\bbfR^3)^3$.
%In other words, we show that  sufficiently small mild solutions of   \rf{NS eq}--\rf{NS div}, are asymptotically  stable  equilibriums  to weak solutions  whose data are   $L^2 (\bbfR^3)^3$ divergence free perturbations of the data of  $V$.

This paper generalizes  the recent work of the first two authors \cite{KarchP}, where the result was restricted to some particular solutions. In \cite{KarchP},  consideration was given to the explicit stationary Slezkin--Landau solutions $V = V(x)$ of the  system \rf{NS eq}--\rf{NS div} \cite{Landau,LL,Slezkin}, which are one-point singular and correspond to  singular external forces. 
A similar technique to that one in \cite{KarchP} has been used in \cite{HHMS} to show the stability of the Ekman spiral, which is an explicit stationary solution to the  three-dimensional Navier-Stokes equations with rotation in the half-space $\R^3_+$ subject to the Dirichlet boundary conditions.

The approach from \cite{KarchP} 
cannot be applied in a general case of a time dependent solution $V = V(x , t)$ (see Remark \ref{rem:stat} for more details). Here, we present a new method which allows to show the $L^2$-stability of a large class of mild solutions including the Slezkin-Landau ones.

Our  result generalizes also a series of papers on the  $L^2$-asymptotic stability either of the zero solution \cite{S1, S, KM, W, BM3, ORS} or  nontrivial stationary solutions \cite{BM95} to the Navier-Stokes system. We only  give an incomplete list of  reference papers  since the entire  list would be overwhelming. The method proposed here allows  to show this type of  asymptotic $L^2$-stability of time-dependent solutions including time-periodic solutions (or almost periodic) and of self-similar solutions,
see Remarks \ref{rem:period} and \ref{rem:self}, below. Limiting our stability results to solutions satisfying the global Serrin 
criterion (see Remark \ref{rem:Kozono}, below), our result relates to the asymptotic stability of large solutions with large
perturbations of the Navier-Stokes equations, obtained by Kozono \cite{K00}.

\subsection*{Notation.}  
\begin{itemize}
   \item We denote by $\|\cdot\|_p$ the usual norm of the Lebesgue space $L^p (\bbfR^3)$ with $p \in [1,\infty]$.
   \item In the case of all other Banach spaces $X$ used in this work, the norm in $X$
is denoted by $\|\cdot\|_X$.
   \item For each space $X$, we set $X_\sigma =\{u\in X^3: \textrm{div } u=0\}$.
   \item $C^\infty_c (\bbfR^3)$ denotes the  set of smooth and compactly supported functions. 
 \item   $\mathcal{S}(\bbfR^3)$ is the Schwartz class of smooth and rapidly decreasing functions.
   \item   The Fourier transform of an integrable function $f$ has the normalization 
\begin{equation*}
    \widehat{f}(\xi ) = (2 \pi)^{-\frac{3}{2}}\int_ {\bbfR^3} \se^{	- i x \cdot \xi }f(x) \ud x,
\end{equation*}
thus $\| f\|_2 = \| \hat{f} \|_2$ for every $f \in L^2 (\bbfR^3)$.
   \item We use the  Sobolev spaces
$H^s(\bbfR^3 ) =\{ f\in L^2(\bbfR^3 ): |\xi|^s \widehat f\in L^2(\bbfR^3 )\}$
and their homogeneous counterparts
$\dot H^s(\bbfR^3 ) =\{ f\in \mathcal{S}'(\bbfR^3 ): |\xi|^s \widehat f\in L^2(\bbfR^3 )\}$ supplemented with the usual norms.
   \item   In the case of $p=2$, the standard inner product in $L^2_{\sigma } (\bbfR^3 )$ is given by $\langle \cdot , \cdot \rangle $. 
 
   \item  Constants independent of solutions
 may change from line to line and  will be denoted by $C$.
\end{itemize} 

%%%%%%%%%%%%%%%%%%%%%%%%%%%%%%%%%%%%%%%%%%%%%%%%%%%%%%%%%%%%%%%%%%%%%%%%%%%%%%%%%%%%%%%%%%%%%%%%%%%%%%%%%%%%%%%%%%%%%%%%%%%%%%%%
\section{Statement of the problem and main results}
\setcounter{equation}{0}

In the sequel, we suppose that $V=V(x,t)$ is a global-in-time solution (in the sense of distributions) to the Navier-Stokes system \rf{NS eq}--\rf{NS in d} with an external force $F=F(x,t)$ and an initial datum $V(x,0)=V_0 (x)$. We require that there exists a Banach space $(\X, \| \cdot \|_{\X})$ such that the solution $V$ satisfies the following properties: 
\begin{itemize}
   \item[{\it i)}]  $V=V(x,t)$ is bounded and weakly continuous in time as a function with values in $\X$:
\begin{equation}\label{V_space}
V \in C_w([0, \infty), \X ),
\end{equation}
that is $V\in L^\infty ([0, \infty ), \X )$ and the function $\langle V(t), \vf \rangle $ is continuous with respect to $t \geqslant 0$ for all $\vf \in \mathcal{X}_\sigma^*$.
   \item[{\it ii)}] There exists a constant $K>0$ such that
\begin{equation}\label{hardy}
   \left| \int_{\bbfR^3} (g\cdot \nabla )h \cdot V(t) \ud x  \right|\leqslant K \sup_{t>0} \| V (t)\|_{\X}\|\nabla g\|_2 \| \nabla h\|_2 ,
\end{equation}
for each $t>0$ and for all $g, h \in \dot H^1_\sigma (\bbfR^3)$ (see Remark \ref{form_b}, below).
   \item[{\it iii)}] The solution $V = V(x,t)$ is sufficiently small, in the sense that 
\begin{equation}\label{K}
   K \sup_{t>0} \| V (t)\|_{\X}<1,
\end{equation}
where $K>0$ is the constant from inequality \rf{hardy}. 
\end{itemize}

In Section \ref{hardy_ineq}, we recall several classical results on  existence of small, global in-time mild solutions to the Navier-Stokes system \rf{NS eq}--\rf{NS div} satisfying assumptions \rf{V_space}--\rf{K}. In particular, we show that the space ${\X}$ can be chosen either as the Lebesgue space $L^3_{\sigma } (\bbfR^3)$, the weak Lebesgue space $L^{3, \infty }_{\sigma } (\bbfR^3 )$,  the Morrey space $\dot{M}^3_p (\bbfR^3) $ for each $2< p \leqslant 3$, or other scaling invariant spaces, see Theorem \ref{thm:hardy}, below.

\begin{rem}\label{form_b}
As it is standard  in the study of the Navier-Stokes system (see {\it e.g.}~\cite{T}), we define the trilinear form  
\begin{equation}\label{b_def}
   \begin{split}
 \mathfrak{b}(f,g,h) &\equiv \sum_{i,j=1}^{3}\int_{\bbfR^3} f^i g_{x_i}^j h^j \ud x = \int_{\bbfR^3} (f\cdot \nabla )g \cdot h \ud x  = -\int_{\bbfR^3} (f \cdot \nabla ) h \cdot g \ud x\\
& = -\mathfrak{b}(f, h, g)
   \end{split}
\end{equation}
for all $f,g,h \in \mathcal{S}_\sigma (\bbfR^3)$. All equalities in \rf{b_def} can be established combining the integration by parts with the divergence free condition. In particular, equality in \rf{b_def} with $g=h$, implies
\begin{equation}\label{b_zero}
   \mathfrak{b}(f, h ,h)=0  \qquad \textrm{for all} \quad f, h \in \mathcal{S}_\sigma (\bbfR^3) .
\end{equation}
Due to \rf{b_def}, our standing assumption \rf{hardy} can be rewritten either as the inequality 
\begin{align}
\label{est_b_1} \big| \mathfrak{b}(g, h, V)\big| &\leqslant K \sup_{t>0} \| V (t)\|_{\X}\|\nabla  g\|_2 \| \nabla  h\|_2  \\
\intertext{or}
\label{est_b_2} \big| \mathfrak{b}(g, V, h)\big| &\leqslant K \sup_{t>0} \| V (t)\|_{\X}\|\nabla  g\|_2 \| \nabla  h\|_2.
\end{align}
for all $g, h \in \dot{H}^1_\sigma (\bbfR^3)$.
\end{rem}

\begin{rem}\label{dual_space}
Notice that inequality \rf{hardy} implies that the mapping $g\mapsto \mathfrak{b}(g, h, V)$ is a bounded linear functional on $\dot{H}^1_\sigma (\bbfR^3)$ for every $h \in \dot{H}^1_\sigma (\bbfR^3)$ and every $V \in \X$. Thus, if $g_n  \rightharpoonup g$ weakly in $\dot{H}^1_\sigma (\bbfR^3)$, then $\mathfrak{b}(g_n, h, V) \rightarrow \mathfrak{b}(g, h, V)$. This observation will allow us to pass to weak limits in the trilinear form $\mathfrak{b}(\cdot , \cdot , \cdot )$.
\end{rem}

\begin{rem}\label{rem:multip}
Following the notation and the terminology from the monograph \cite[Ch.~21]{Lemarie}, inequality \rf{hardy} holds true if $\X\subset X_1(\bbfR^3)$, where 
$X_1(\bbfR^3)$ is the set of pointwise multipliers from $H^1_\sigma (\bbfR^3)$ to $L^2_\sigma(\bbfR^3)$. The linear space $X_1(\bbfR^3)$ is a Banach space equipped 
with the norm $ \|f\|_{X_1}\equiv \sup \{ \|fg\|_2\,:\, \|g\|_{H^1}\leq 1\}.$
It is easy to show using a duality argument in the Hilbert space $L^2(\bbfR^3)$ (see {\it e.g.} \cite[Ch.~2]{LemGal}) that the inequality 
\begin{equation*}
   \bigg|\int_{\bbfR^3} W\cdot ( g \cdot \nabla ) h \ud x \bigg|\leqslant C
\|\nabla  g \|_2 \| \nabla  h  \|_2
\end{equation*}
holds true for all vector fields $g, h \in \dot H^1 (\bbfR^3)$  and a certain constant $C=C(W)$ if and only if $W \in X_1(\bbfR^3)$.
We refer the reader to \cite{Lemarie, LemGal, SD,G06} for more properties of pointwise multipliers and, in particular, for explanations (as well as an extensive review and a complete bibliography) how  inequality \rf{hardy} is related to the so-called weak-strong uniqueness of solutions to the Navier-Stokes equations.
\end{rem}

We now state the main results of this work: a type of  asymptotic stability for  global-in-time solutions $V=V(x,t)$ under arbitrary 
large  $L^2_{\sigma }(\bbfR^3 )$-perturbations. 

\begin{theorem}[Existence of weak solutions]\label{main}
Let  $V=V(x,t)$ be a global-in-time solution to the initial value problem  \rf{NS eq}--\rf{NS in d} in $C_w\big( [0, \infty ) , \X \big)$ satisfying properties  \rf{V_space}--\rf{K}. Denote $V_0 = V(\cdot , 0)$ and let $w_0 \in L^2_\sigma (\bbfR^3)$ be arbitrary. 
Then, the Cauchy problem \rf{NS eq}--\rf{NS in d} with the initial condition $u_0=V_0 + w_0 $ and  the same external force $F$ has a global-in-time distributional solution $u = u(x,t)$ of the  form $u(x,t)=V(x,t)+w(x,t)$, where $w = w(x,t)$ is a weak solution of the corresponding  perturbed problem {\rm (}see \rf{w:eq}--\rf{w:ini} below{\rm )} satisfying
\begin{equation}\label{ex_w_space}
  w\in  X_T \equiv  C_w \big( [0,T], L^2_\sigma (\bbfR^3)\big)\cap L^2\big([0,T], \dot H_\sigma^1(\bbfR^3)\big)  \qquad \textit{for each} \ \ T>0.
\end{equation}
\end{theorem}

\begin{theorem}[Asymptotic behavior of weak solutions]\label{main2}
A solution $u = u(x,t)$ of problem \rf{NS eq}--\rf{NS in d} considered in Theorem \ref{main} can be constructed so that  \[ \| w(t) \|_2 = \| u(t) - V(t) \|_2 \rightarrow 0 \qquad \mbox{as}\quad t\rightarrow \infty .\]
\end{theorem}

For the proofs of Theorems \ref{main} and \ref{main2},  denote by $u=u(x,t)$ a solution of the Navier--Stokes system \rf{NS eq}--\rf{NS in d} with  external force $F=F(x,t)$ and  initial data $u_0=V_0 +w_0$, where $w_0\in L^2_\sigma (\bbfR^3)$. Then, the functions $w(x,t)=u(x,t)-V(x,t)$ and $\pi (x,t)=p(x,t)-p_V(x,t)$ (here, $p_V$ is a pressure associated with the velocity field $V$) satisfy the perturbed  initial value problem
\begin{align}
   \label{w:eq} w_t-\Delta w+(w\cdot  \nabla )w+ (w \cdot \nabla )V + (V \cdot \nabla )w +\nabla \pi &=0,\\
   \label{w:d}   {\rm div}\  w&=0,\\
   \label{w:ini} w(x,0)&=w_0(x)\ .
\end{align} 

Thus, our main goal  is to construct a weak solution $w$ of problem \rf{w:eq}--\rf{w:d} and  show its $L^2$-decay to zero as  $t \rightarrow \infty $. First, recall   the following  standard definition.
\begin{definition}\label{weak_sol}
A vector field $w = w(x,t)$ is called a weak solution to problem \rf{w:eq}--\rf{w:ini} if it belongs to the classical energy space  
\begin{equation}\label{XT}
  X_T= C^{\infty }_w \big( [0,T], L^2_\sigma (\bbfR^3)\big)\cap L^2\big([0,T], \dot H_\sigma^1(\bbfR^3)\big) 
\end{equation}
and if 
\begin{equation}\label{weak}
    \begin{split}
     \big\langle  w(t), \vf (t)  \big\rangle +&\int_s^t \Big[  \big\langle \nabla w, \nabla \vf \big\rangle + \big\langle w\cdot  \nabla w, \vf \big\rangle - \big\langle ( w \cdot  \nabla )\vf , V \big\rangle + \big\langle (V \cdot \nabla )w, \vf \big\rangle  \Big] \ud \tau \\
 &= \big\langle w(s), \vf (s)\big\rangle + \int_s^t \big\langle w, \vf_\tau \big\rangle \ud \tau
    \end{split}
\end{equation}
for all $t \geqslant s \geqslant  0$ and all $\vf \in C \big([0, \infty), H_\sigma^1(\bbfR^3)\big) \cap C^1\big([0,\infty), L^2_\sigma (\bbfR^3)\big)$, where $\langle\cdot, \cdot \rangle $ is the  inner product in $L^2_{\sigma } (\bbfR^3 )$. 
\end{definition}

Theorems \ref{main} and \ref{main2} are immediate consequences of the following result.

\begin{theorem}\label{th:ex}
For every $w_0 \in L^2_\sigma (\bbfR^3 )$  and each $T>0$, problem \rf{w:eq}--\rf{w:ini} has a weak solution  $w \in {X}_T$ $ ($cf. \rf{XT}$)$ for which  the strong energy inequality
\begin{equation}\label{en:inq}
   \|w(t)\|_2^2+ 2\big( 1-K\sup_{t>0}\| V(t)\|_{\X}\big) \int_s^t \|\nabla w(\tau )\|_2^2 \ud \tau \leqslant \|w (s)\|_2^2
\end{equation}
holds true for almost all $s\geqslant 0$, including $s=0$ and all $t\geqslant s$ and, which satisfies \[\lim_{t\rightarrow \infty } \|w(t)\|_2 = 0.\]
\end{theorem}

\begin{rem}
Assuming that $w_0 \in L^2_\sigma (\bbfR^3 ) \cap L^p(\bbfR^3)$ with some $1\leq p<2$, we expect an algebraic decay rate of the quantity $\|w(t)\|_2$ as $t\to\infty$ as in the case of the $L^2$-decay of weak solutions to the Navier-Stokes equations
(see {\it e.g.}~\cite{S,BM3,ORS}). We do not attempt to make such improvements since they are more-or-less standard.
\end{rem}

The proof of Theorem \ref{th:ex} will be split into two parts corresponding to Theorem \ref{main} and Theorem \ref{main2}, which will be developed in Sections \ref{existence} and \ref{splitting}, respectively. In Section ~\ref{existence}, using Galerkin approximations, we show the existence of a solution to equation \rf{weak} satisfying strong energy inequality \rf{en:inq}. 
Next,  in Section \ref{sec:GEI}, we show  that this solution satisfies a class of general energy estimates. In Section \ref{splitting},  suitable test functions are used in  generalized energy inequalities combined with a modified  Fourier splitting technique \cite{ORS} to yield the convergence $\| w (t) \|_2 \rightarrow 0$ as $t \to \infty$.  

First, however, we discuss some consequences of Theorems \ref{main} and  \ref{main2}. In the following series of remarks we describe classes of global-in-time mild solutions, which are $L^2$-globally asymptotically stable in the sense discussed above.

\begin{rem}\label{rem:stat}
If an external force $F$ in equations \rf{NS eq} is independent of time, one may expect system \rf{NS eq}--\rf{NS div} to have stationary solutions. This is indeed the case and there are several results on the existence of small stationary solutions in scaling invariant spaces, see {\it e.g.}~\cite{KY95s,Yamazaki,BM95,CanKarch,CK05,KarchP, BBIS11}.
By Theorems \ref{main} and \ref{main2}, if these stationary solutions belong to a Banach space $\X$ and  satisfy our standing assumptions \rf{hardy}--\rf{K},  they are asymptotically stable under arbitrary large $L^2$-perturbations.

In the case when  $V$ is time-independent,  the linearization at zero solution of the perturbed problem \rf{w:eq}--\rf{w:ini} generates an analytic semigroup of linear operators on $L^2_\sigma(\bbfR^3)$, see \cite[Ch.~4]{KarchP} for an example of such a reasoning in the case of singular stationary solutions. This observation allows to apply ideas of Borchers and Miyakawa \cite{BM3,BM95} to show the decay of $\|w(t)\|_2$. This approach cannot be used in the case of time dependent coefficients in  problem \rf{w:eq}--\rf{w:ini}, hence we needed to develop a new technique to overcome  technical obstacles in the proof of Theorem~\ref{main2}.
\end{rem}

\begin{rem}\label{rem:period}
If the external force  $F=F(x,t)$ is small in a suitable sense and time-periodic (or almost periodic in $t$), then the Navier-Stokes problem \rf{NS eq}--\rf{NS in d} has a unique global-in-time mild solution in the space $C_w([0,\infty), L^{3,\infty}_\sigma(\bbfR^3))$. This solutions is  time-periodic (almost periodic in $t$, respectively), see \cite[Cor.~1.2]{Yamazaki} 
and Subsection \ref{subsec:M} for results and references on the existence of global-in-time solutions to the Navier-Stokes equations in the Marcinkiewicz space $L^{3,\infty}_\sigma(\bbfR^3)$. Choosing $\X = L^{3, \infty }_\sigma (\bbfR^3)$ in Theorems \ref{main} and \ref{main2} we obtain that, for sufficiently small external forces, these solutions are also asymptotically stable under arbitrary large divergence free initial $L^2$-perturbations.
\end{rem}

\begin{rem}\label{rem:self}
Suppose now that an initial datum $V_0\in L^{3,\infty}_\sigma(\bbfR^3) $ (or, more generally, in a suitable Morrey spaces) is sufficiently small and homogeneous of degree $-1$. Then, the Navier-Stokes problem
\rf{NS eq}--\rf{NS in d} has a global-in-time mild {\it self-similar} solution {\it i.e.}~a solution of the form
\begin{equation*}
     V(x,t) =\frac{1}{\sqrt{t}}V\left(\frac{x}{\sqrt{t}}\right) \in C_w\big([0,\infty), L^{3,\infty}_\sigma(\bbfR^3)\big),
\end{equation*}
see \cite{Cann-Book,Cannone,Lemarie} for a review of the theory of self-similar solutions to the Navier-Stokes equations.
By Theorems \ref{main} and \ref{main2}, 
problem \rf{NS eq}--\rf{NS in d}  with  the initial condition $u_0=V_0+w_0$, where  $w_0\in L^2_\sigma(\bbfR^3)$ is arbitrary, and  with a suitable and small external force has a global-in-time distributional solution of the form $u(x,t)=t^{-1/2}V(x/\sqrt{t})+w(x,t)$, where 
$\|w(t)\|_2\to 0$ as $t\to\infty$.
\end{rem}

Our main result combined with  Kato's  theorems \cite[Thm.~4 and 4']{Kato} yields the following corollary.

\begin{corollary}\label{cor:Cald}
Fix $p\in [2,3)$. For every $u_0\in L^p_\sigma(\bbfR^3)$ and $F\equiv 0$, the Navier-Stokes problem  \rf{NS eq}--\rf{NS in d}
has a global-in-time {\rm (}distributional{\rm )} solution $u=u(x,t)$. This solution can be written in the form $u(x,t)=V(x,t)+w(x,t)$
where $V\in C([0,\infty), L^3_\sigma(\bbfR^3))$ is a mild solution of the Navier-Stokes problem \rf{NS eq}--\rf{NS in d} 
and $w\in X_T$ for each $T>0$ {\rm (}see \rf{XT}{\rm )}
is a weak solution of the perturbed problem \rf{w:eq}--\rf{w:ini}. These two vector fields satisfy
\begin{equation}\label{dec:Cald}
      \|V(t)\|_3\to 0\quad\text{and}\quad \|w(t)\|_2\to 0 \qquad \text{as}\quad t\to \infty.
\end{equation}
\end{corollary}

The existence of  global-in-time weak solutions to problem
 \rf{NS eq}--\rf{NS in d},  with  $u_0\in L^p_\sigma(\bbfR^3)$ and  $p\in [2,3)$, was established by  Calder\'on \cite{C90}. Here, we improve Calder\'on's result by showing   
that  these solutions decay as $t \to \infty$ in the sense expressed in \rf{dec:Cald}. The proof of this corollary is given in Subsection \ref{subsec:Lp} after discussing  Hardy type inequality \rf{hardy} for  $L^p$-spaces.

\begin{rem}\label{rem:Kozono}
Under the assumption that the solution $V=V(x,t)$ belongs to the Serrin class: $V\in L^\alpha(0, \infty; L^q(\R^3))$ 
 for $2/\alpha+3/q=1$ with $3<q\leq \infty$, its global asymptotic $L^2$-stability was shown by Kozono \cite{K00}.
Corollary \ref{cor:Cald} can be considered as an extension of the Kozono result to the limit case $q=3$ and $\alpha=\infty$.
\end{rem}

In the next section, we explore  examples of   norm-scale-invariant Banach spaces $\X$, that is spaces satisfying $\| \lambda f(\lambda \cdot )\|_{\mathcal X} =\|f\|_{\mathcal X}$ for all $f \in \X$ and for all $\lambda >0$, for which inequality \rf{hardy} (or, equivalently, inequalities \rf{est_b_1}, and \rf{est_b_2}) holds true. We also recall results on the global-in-time existence of mild solutions $V\in C_w\big([0,\infty), \X\big)$ to problem \rf{NS eq}--\rf{NS in d} which satisfy our standing assumptions \rf{V_space}--\rf{K}.

%%%%%%%%%%%%%%%%%%%%%%%%%%%%%%%%%%%%%%%%%%%%%%%%%%%%%%%%%%%%%%%%%%%%%%%%%%%%%%%%%%%%%%%%%%%%%%%%%%%%%%%%%%%%%%%%%%%%%%%%%%%%%%%%
\section{Hardy-type inequalities}\label{hardy_ineq}
\setcounter{equation}{0}

The following theorem contains a brief introduction to the theory of global-in-time mild solutions to the Navier-Stokes 
problem \rf{NS eq}--\rf{NS in d}. These solutions satisfy assumptions \rf{V_space}--\rf{K}, hence,
are they globally asymptotically stable under arbitrary $L^2$-perturbation in the sense described in Theorems \ref{main} and \ref{main2}. 

\begin{theorem}\label{thm:hardy}
There exists a constant $K>0$ such that  the following inequality holds true
\begin{equation}\label{hardy1}
   \bigg|\int_{\bbfR^3} W\cdot ( g \cdot \nabla ) h \ud x \bigg|\leqslant K \| W\|_{\X}\|\nabla  g \|_2 \| \nabla  h  \|_2,
\end{equation}
for all vector fields $g, h \in \dot H^1 (\bbfR^3)$ and all $W \in \X$, where $\X$ is one of the Banach spaces
\begin{itemize}
  \item $\X = \dot{H}^{{1}/{2}}_\sigma (\bbfR^3 )$ $($the homogeneous Sobolev space$)$,
  \item $\X = L^3_\sigma (\bbfR^3)$ $($the Lebesgue space$)$,
  \item $\X= \{ f \in L^\infty _{loc}(\bbfR^3 ): \|f\| = \sup_{x\in \X}|x| |f(x)| <\infty \}$ $($the weighted $L^\infty $-space$)$,
  \item $\X= {\mathcal{PM}}^2(\bbfR^3)$ $($the Le Jan-Sznitman space$)$,
  \item $\X= L^{3, \infty }_\sigma (\bbfR^3)$ $($the Marcinkiewicz space$)$,
  \item $\X = \dot{M}^3_p (\bbfR^3) $ for each $2< p \leqslant 3$ $($the Morrey space$)$.
 \end{itemize} 
In the case of each space $\X$, there exist constants  $\varepsilon >0$ and  $C>0$ such that 
for all $V_0\in \X$ such that $\|V_0\|_{\X}<\varepsilon$, the Navier-Stokes problem 
\rf{NS eq}--\rf{NS in d} with the initial datum $u_0=V_0$ and the external force $F\equiv 0$ has a global-in-time solution in the space $C_w([0,\infty),\X)$ which satisfies $\sup_{t>0}\|V(t)\|_{\X}\leq 
C\|V_0\|_{\X}$.
\end{theorem}

We note that the proof of inequality \rf{hardy1}  is more-or-less standard (see Remark \ref{rem:multip}) and we sketch it for 
the completeness of the exposition.
Let us recall  the well-known continuous  embeddings 
\begin{equation*}
  \dot{H}^{{1}/{2}}_\sigma (\bbfR^3 )\subset
L^3_\sigma (\bbfR^3) \subset
 L^{3, \infty }_\sigma (\bbfR^3) \subset
 \dot{M}^3_p (\bbfR^3) \qquad  \text{for each} \quad 2< p \leqslant 3
\end{equation*}
as well as 
\begin{equation*}
   \{ f \in L^\infty _{loc}(\bbfR^3 ): \|f\| = \sup_{x\in \X}|x| |f(x)| <\infty \}  \subset
 L^{3, \infty }_\sigma (\bbfR^3)\quad\text{and}\quad
{\mathcal{PM}}^2(\bbfR^3)
\subset
 L^{3, \infty }_\sigma (\bbfR^3),
\end{equation*}
{\it cf.}~Remark \ref{rem:PM}.
Thus, it  would suffice to prove inequality \rf{hardy1}  in the case of the Morrey space
$ \dot{M}^3_p (\bbfR^3)$, solely. Nevertheless, we discuss separately each of the spaces embedded in the Morey, because this approach  emphasizes the  relations of our theorems with classical results on the global-in-time existence of small mild solutions to the Navier-Stokes equations.

Here, for simplicity of the exposition, we consider the Navier-Stokes problem \rf{NS eq}--\rf{NS in d} with $F\equiv 0$, however, 
global-in-time small mild solutions exist also
in the case of  small non zero external forces, see {\it e.g.}~the publications \cite{CannonePlan,Yamazaki,CanKarch,CK05} and references therein. 
In the following subsections,  other  references are provided  for results on the existence of solutions to the Navier-Stokes problem \rf{NS eq}--\rf{NS in d} in $C_w\big( [0, +\infty ) , \X \big)$, where $\X$ is one of the spaces defined in Theorem \ref{thm:hardy}. Our list of references is far from being complete since the existing literature is extensive.

\subsection{Homogeneous Sobolev space}

For $\X = \dot{H}^{{1}/{2}}_\sigma(\bbfR^3) $, the proof of the inequality 
\begin{equation*}
  \bigg|\int_{\bbfR^3} W\cdot ( g \cdot \nabla ) h \ud x \bigg|\leqslant K \| W\|_{\dot{H}^{{1}/{2}}_\sigma}\|\nabla  g \|_2 \| \nabla  h  \|_2
\end{equation*}
for all $f, g \in  \dot H^1 (\bbfR^3)$ and $W \in \dot{H}^{{1}/{2}}(\bbfR^3) $  can be found in a much more general setting 
{\it e.g.}~in \cite[Lemma 1]{LemGal}. Under suitable smallness assumptions on $u_0\in  \dot{H}^{{1}/{2}}_\sigma(\bbfR^3)$ with $F\equiv 0$, Fujita and Kato 
\cite{FK64} obtained existence of  solutions to the Navier-Stokes problem \rf{NS eq}--\rf{NS in d} in $C([0,\infty), \dot{H}^{{1}/{2}}_\sigma(\bbfR^3))$, satisfying  conditions \rf{V_space}--\rf{K}.

\subsection{Lebesgue space}\label{subsec:Lp}
In the case of the Lebesgue space $\X = L^3_\sigma (\bbfR^3)$, the H\"older and the Sobolev inequalities yield
\begin{equation*}
   \bigg| \int_{\bbfR^3}W\cdot (g \cdot \nabla ) h  \ud x \bigg|  \leqslant \| W\|_3 \| g \|_6\| \nabla  h\|_2  \leqslant K \| W\|_3 \| \nabla  g \|_2 \| \nabla  h\|_2 ,
\end{equation*}
for all $f, g \in  \dot H^1 (\bbfR^3)$ and $W \in L^3_\sigma (\bbfR^3)$.
The existence of a global-in-time solution to the Cauchy problem \rf{NS eq}--\rf{NS in d} with $u_0\in  L^3_\sigma (\bbfR^3)$ and 
$F\equiv 0$ was established by Kato \cite{Kato}. Kato's solutions  satisfy  suitable decay estimates, hence Theorems \ref{main} and \ref{main2} yield  our Corollary \ref{cor:Cald}, which  improves  Calder\'on's result \cite{C90}.
The proof of 
Corollary \ref{cor:Cald}  is as follows:

\begin{proof}[Proof of Corollary \ref{cor:Cald}] Let $p\in [2,3)$ and   $u_0\in L^p_\sigma(\bbfR^3)$.
For each  constant $R>0$  define
\begin{equation*}
u_{0,R}(x)=
\left\{
\begin{array}{ccc}
u_0(x)&\text{if}& |u_0(x)|\leq R\\
R&\text{if}& |u_0(x)|> R
\end{array}
\right.
\qquad\text{and}\qquad
u_0^R(x)=
\left\{
\begin{array}{ccc}
0&\text{if}& |u_0(x)|\leq R\\
u_0(x)-R&\text{if}& |u_0(x)|> R
\end{array}
\right.
\end{equation*}
Write    $u_0=V_0+w_0, $ where $V_0=\P (u_{0,R})$ and $w_0=\P (u_0^R)$. Here, $\P: L^2(\R^3)^3\to  L^2_\sigma(\R^3)$
is the Leray orthogonal projection on divergence free vector fields. This operator  can be extended to a bounded operator on $L^p(\R^3)^3$ for each $1<p<\infty$.

 Notice that  $V_0\in L^q_\sigma(\bbfR^3)$ for each $q\in [p,\infty)$. 
Apply  Kato's theorems  \cite[Thm.~4 and 4']{Kato} with sufficiently 
small  $R>0$ to obtain a global-in-time mild solution $V\in C([0,\infty), L^3_\sigma(\bbfR^3))$ 
of    \rf{NS eq}--\rf{NS in d}. This solution satisfies the standing assumptions 
\rf{V_space}--\rf{K} with $\X=L^3_\sigma(\bbfR^3)$ and  decays according to  the first relation in \rf{dec:Cald}.
Notice that  the truncated vector field $u_0^R$  is nonzero on a set of finite measure  hence $w_0=\P(u_0^R)\in L^2(\bbfR^3)$. 
The existence of a decaying solution $w=w(x,t)$ of the perturbed problem \rf{w:eq}--\rf{w:ini}  follows now by Theorems \ref{main} and \ref{main2}.
\end{proof}

\subsection{Weighted $L^\infty $-space}

Let $\X = \{ f \in L^\infty _{loc }(\bbfR^3 ): \|f\|_{\X } = \sup_{x\in \X}|x| |f(x)| <\infty \}$. Combine the Schwarz inequality, the definition of the space $\X$, and the following Hardy inequality ({\it cf.} \cite[Eq.~(1.14)]{L})
\begin{equation*}
    \int_{\bbfR^3} \frac{ |g(x)|^2}{|x|^2} \ud x \leqslant 4 \int_{\bbfR^3} |\nabla g|^2 \ud x \qquad \text{for all}\quad g \in  H^1 (\bbfR^3)
\end{equation*}
to  yield 
\begin{align*}
   \bigg| \int_{\bbfR^3}W\cdot ( g \cdot \nabla ) h \ud x \bigg|  &\leqslant \| W g \|_2\| \nabla  h \|_2  \leqslant \| W\|_{\X}\|  |\cdot |^{-1} g \|_2 \| \nabla h \|_2 \\
&\leqslant K\| W\|_{\X} \|\nabla  g \|_2 \| \nabla h \|_2
\end{align*}
for all $f, g \in  \dot H^1 (\bbfR^3)$ and $W \in \X$. The existence of global-in-time small mild solutions to the 
the Navier-Stokes problem  \rf{NS eq}--\rf{NS in d} in $\X$ was established by Cannone \cite[Sec.~4.3.1]{Cann-Book}, and Cannone and Planchon \cite{CP96}. Improvements of this result may be found in \cite{Mi00,Mi02,BV07}. Since  $\X\subset L^{3, \infty } (\bbfR^3)$, we postpone the further discussion of solutions to problem \rf{NS eq}--\rf{NS in d} in these spaces to Subsection \ref{subsec:M}.

\subsection{Le Jan-Sznitman space}
 The existence of global-in-time small singular solutions to the incompressible 
Navier-Stokes system with singular external forces  was established in  \cite{CanKarch}. In  \cite{CanKarch},  the authors also  studied the  asymptotic stability of  solutions in the  Banach  spaces 
\begin{equation*}
    {\mathcal{PM}}^2= \{ v\in \mathcal{S}' (\bbfR^3) : \widehat v \in L^1_{loc} (\bbfR^3) , \quad \| v\|_{\mathcal{PM}^2}=\textrm{ess}\sup_{\xi \in \bbfR^3} |\xi |^2 |\widehat v(\xi )| <\infty \}.
\end{equation*}
In particular,  the Slezkin--Landau singular stationary solutions,  which 
have been considered in \cite{KarchP} belong to this space.

Our results on the $L^2$-global asymptotic stability can be applied to solutions solutions from $C_w\big([0,\infty), {\mathcal{PM}}^2\big)$, as well.
Indeed, let us prove inequality \rf{hardy1} with $\X = \mathcal{PM}^2(\bbfR^3)$.  
Properties of the Fourier transform combined with the H\"older inequality in the Lorentz spaces (see {\it e.g.}~\cite{O'N}) yield 
\begin{align*}
   \bigg| \int_{\bbfR^3}W\cdot ( g \cdot \nabla ) h  \ud x \bigg|  &=\bigg| \int_{\bbfR^3}\widehat {W(\xi )}\cdot (\overline{\widehat{g \cdot \nabla h}})(\xi ) \ud \xi \bigg| \\
   &\leqslant \| W\|_{{\mathcal{PM}}^2}\int_{\bbfR^3} |\xi |^{-2}\big| \widehat{ g \cdot \nabla h}(\xi ) \big| \ud \xi \\
   &\leqslant \| W\|_{{\mathcal{PM}}^2}\| |\cdot |^{-2}\|_{L^{\frac{3}{2},\infty }}\| \widehat{g \cdot \nabla h }\|_{L^{3,1}}.
\end{align*}
It suffices to apply the Hausdorff-Young inequality in the Lorentz spaces \cite{K}:
\begin{equation*}
   \|\widehat f\|_{L^{p',r}} \leqslant C\| f\|_{ L^{p,r}}, \qquad \textrm{where}\ \ 1<p<2, \ 1\leqslant r<\infty , \ \textrm{and }\ \ \frac{1}{p'}+\frac{1}{p}=1, 
\end{equation*}
together with  the H\"older and Sobolev inequalities for Lorentz spaces. Hence it follows
\begin{equation}\label{PM:ineq}
   \begin{split}
   \bigg| \int_{\bbfR^3}W\cdot ( g \cdot \nabla ) h  \ud x \bigg|  &\leqslant C\| W\|_{{\mathcal{PM}}^2} \| g  \cdot \nabla h\|_{L^{{3}/{2}, 1}}\\
   & \leqslant C \| W\|_{{\mathcal{PM}}^2} \| g \|_{L^{6,2} }\|\nabla h \|_{L^{2,2}} \\
   &\leqslant K\| W\|_{{\mathcal{PM}}^2}\|\nabla  g \|_2 \| \nabla h \|_2 
  \end{split}
\end{equation}
for all $g, h \in \dot{H}^1 (\bbfR^3)$ and $W \in \mathcal{PM}^2(\bbfR^3)$.

\begin{rem}\label{rem:PM}
As a byproduct of the estimates  in this subsection, we obtain 
the following  short proof of the embedding 
 $\mathcal{PM}^2(\bbfR^3) \subset L^{3,\infty}(\R^3)$. Repeating the reasoning which leads to the first inequality in
\rf{PM:ineq} we find  a constant $C>0$ such that for all $W\in \mathcal{PM}^2(\bbfR^3)$ and all $\varphi\in L^{3/2,1}(\R^3)$
we have \[ \left|\int_{\R^3} W\varphi \ud x\right| \leqslant C\| W\|_{{\mathcal{PM}}^2} \| \varphi\|_{L^{{3}/{2}, 1}}.\]  Hence, the distribution $W$ defines a continuous linear functional on 
$ L^{3/2,1}(\R^3)$ and, as a consequence, we have   $W\in  L^{3,\infty}(\R^3)$.

\end{rem}

\subsection{Marcinkiewicz space}\label{subsec:M}
The proof of  inequality \rf{hardy1} with $\X = L^{3, \infty } (\bbfR^3)$ involves  the  H\"older  and  Sobolev  inequalities in the Lorentz $L^{p,q}$-spaces (see {\it e.g.} \cite{O'N} and \cite{B}, resp.) as follows
\begin{align*}
   \bigg| \int_{\bbfR^3}W\cdot ( g \cdot \nabla ) h \ud x \bigg|  &\leqslant C\| W g \|_2\| \nabla h \|_2  \leqslant C\| W g \|_{L^{2,2}}\|\nabla h\|_2 \\
&\leqslant C\| W\|_{L^{3,\infty }}\| g \|_{L^{6,2}}\| \nabla  h\|_2 \\
&\leqslant C\| W\|_{L^{3,\infty }}\| \nabla  g \|_2 \| \nabla  h\|_2
\end{align*}
for all $f, g \in  \dot H^1_\sigma (\bbfR^3)$ and $W \in L^{3 ,\infty} (\bbfR^3)$
(this proof may be also found {\it e.g.}~in \cite[Proposition~3.4 in  the case $p=3$]{SD}).

Global-in-time mild solutions to the Navier-Stokes with small initial conditions in  $L^{3, \infty }_\sigma (\bbfR^3)$
have been constructed by Kozono and Yamazaki \cite{KY95} (see also Barraza \cite{B96}). These solutions are unique in the space
$C_w([0,\infty), L^{3,\infty}_\sigma (\bbfR^3))$ intersected with a set of functions with appropriate decay in time. The construction of \cite{KY95}   was improved by Meyer  \cite{Meyer}, who applied the Banach contraction principle to obtain non-decaying solutions in the space $C_w([0,\infty), L^{3,\infty}_\sigma (\bbfR^3))$. An analogous argument was used by Yamazaki  \cite{Yamazaki} to deal with the Navier-Stokes equations with time-dependent external forces in the whole space,  the half-space and,  exterior domains. Yamazaki formulated sufficient  conditions on the initial conditions and external forces to insure the existence of  unique small solutions bounded for all time in  weak $L^3$-spaces. Stability properties of small mild global-in-time solutions to the Navier-Stokes problem with external forces has been also studied in \cite{CK05}.

\subsection{Morrey space}

Let $1 < p \leqslant q < \infty$ and $p'=p/(p-1)$ and $q'=q/(q-1)$.
The homogeneous Morrey spaces are defined as
\begin{equation*}
    \dot{M}^q_p(\bbfR^3) = \{ f\in L^q_{loc}(\bbfR^3): \|f\|_{\dot{M}^q_p}= \sup_{R>0}\sup_{x\in \bbfR^3}\Big(\int_{B_R(x)}|f(y)|^p \ud y\Big)^{\frac{1}{p}}<\infty\}.
\end{equation*}
It is known (see \cite[Lemma 21.1]{Lemarie} and \cite[Lemma 3.13]{SD}) that this space is the dual space of  $ \dot{N}^{q'}_{p'}(\bbfR^3 )$ defined in the following way. For $1 < q' \leqslant p' < \infty $, we set
\begin{equation*}
     \dot{N}^{q'}_{p'}(\bbfR^3 ) = \left \{ \begin{aligned}
 f \in L^{q'}(\bbfR^3 ) : \ f &\stackrel{L^{q'}}= \sum_{k\in \bbfN}g_k, \quad \textrm{where} \quad \{g_k\} \subset L^{p'}_{comp}(\bbfR^3) \quad \textrm{and} \\
\sum_{k\in \bbfN} d^{n\big(\frac{1}{q'} - \frac{1}{p'}\big)}_k \| g_k \|_{L^{p'}} &<\infty , \quad  \textrm{where}\quad  \forall k \quad  d_k = \textrm{diam}(\textrm{supp}\  g_k ) <\infty
 \end{aligned}
 \right \}
\end{equation*}
which is a Banach space  equipped with the norm 
\begin{equation*}
  \| f \|_{\dot{N}^{q'}_{p'}} = \inf\Big\{ \sum_{k\in \bbfN} d_k^{n\big(\frac{1}{q'}\frac{1}{p'}\big)}\|g_k \|_{p'}\Big\}.
\end{equation*}
We also recall an estimate for the product  functions in $\dot{N}^{q'}_{p'}(\bbfR^3)$.  This estimate will be essential for   the proof of our  of inequality \rf{hardy1} in  $\X = \dot{M}^3_p(\bbfR^3)$. 

\begin{lemma}[{\cite[Prop. 21.1]{Lemarie}} and {\cite[Lemma 3.14]{SD}}]
\label{hol_mor}
Let $1 \leqslant q' \leqslant p' \leqslant 2$. There exists 
a constant 
$C>0$ such that for all $f \in L^2(\bbfR^3)$ and $g\in \dot{H}^{3/q}(\bbfR^3)$,
\begin{equation}\label{hol_mor_eq}
    \| fg\|_{\dot{N}^{q'}_{p'}}\leqslant C \| f\|_2 \| g\|_{\dot{H}^{3/q}}.
\end{equation}    
\end{lemma} 

The proof of inequality \rf{hardy1} follows easily by   inequality
\rf{hol_mor_eq} with $q=3$

\begin{align*}
    \bigg| \int_{\bbfR^3}W\cdot ( g \cdot \nabla ) h  \ud x \bigg|  & \leqslant \| W\|_{\dot{M}^3_p}\| g \nabla h\|_{\dot{N}^{\frac{3}{2}}_{p'}} \\
    &\leqslant  K \| W\|_{\dot{M}^3_p} \| \nabla g\|_2 \| \nabla h\|_2 ,
\end{align*}
for all $f, g \in  \dot H^1_\sigma (\bbfR^3)$ and $W \in \dot{M}^3_p (\bbfR^3)$, where $2<p \leqslant 3$ and $p$ and $p'$ are conjugate exponents.

 For   constructions of global-in-time mild solutions to the Navier-Stokes problem \rf{NS eq}--\rf{NS in d} in Morrey spaces,  
we refer the reader to  \cite{GM89, K92} as well as to the book  \cite[Ch.~18]{Lemarie} (and to references therein)
For small initial data, these solutions satisfy our standing assumptions \rf{V_space}--\rf{K} with $\X=\dot{M}^3_p (\bbfR^3)$, where $2<p \leqslant 3$.

%%%%%%%%%%%%%%%%%%%%%%%%%%%%%%%%%%%%%%%%%%%%%%%%%%%%%%%%%%%%%%%%%%%%%%%%%%%%%%%%%%%%%%%%%%%%%%%%%%%%%%%%%%%%%%%%%%%%%%%%%%%%%%%%
\section{Existence of weak solutions}\label{existence}
\setcounter{equation}{0}

In this section, we construct weak solutions to the perturbed initial value problem \rf {w:eq}--\rf{w:ini}.  

\begin{proof}[Proof of Theorem \ref{th:ex}. First part -- Existence of solutions]
Step one  is a construction of
 weak solutions to problem \rf{w:eq}--\rf{w:ini} satisfying the strong energy inequality \rf{en:inq}. It follows a relatively standard Galerkin technique (see {\it e.g.} \cite[Ch. III. Thm. 3.1]{T}). 

Let $\{g_m\}^\infty _{m=1}$ be an orthonormal complete system in $L^2_\sigma (\bbfR^3)$ and  assume that $g_m \in H^1_\sigma (\bbfR^3)$ for every $m \in \bbfN$. Let $W_m$ be the  linear space spanned by 
$\{  g_1, g_2, ..., g_m \}$ for each $m=1, 2, ...$. Define an approximate solution $w_m : [0,T] \to W_m$ by 
\begin{equation*}
    w_m (t) = \sum_{i=1}^{m}d_{im}(t)g_i,
\end{equation*}
where the  coefficients $d_{im}=d_{im} (t)$ satisfy 
\begin{equation}\label{w m}
   \begin{split} 
    \frac{\ud }{\ud t} d_{jm}(t)&=\frac{\ud }{\ud t}\langle w_m(t),g_j\big\rangle  = \big\langle \nabla w_m (t), \nabla g_j\big\rangle + \mathfrak{b}\big( w_m(t) ,w_m(t),g_j\big) \\
&+ \mathfrak{b}\big( w_m(t),V(t) , g_j \big) + \mathfrak{b}\big( V(t) , w_m(t) ,g_j\big) =0 \qquad \text{for}\quad j=1,...,m , \\
   w_m (0)&= P_m w_0,
   \end{split}
\end{equation}
with the orthogonal projection operator $P_m: L^2_\sigma (\bbfR^3)\to W_m$ is given by  $P_m (v) = \sum_{i = 1}^m  \langle v, g_i \rangle g_i $. Recall that the term corresponding to the pressure in \rf{w:eq} vanishes in \rf{w m} since $\big\langle \nabla \pi, g_j \big\rangle = 0 $ as ${\rm div }\ g_j =0$. 

Due to  assumption \rf{hardy} (\textit{cf.}~Remark \ref{form_b}) both terms in \rf{w m} containing the solution $V(t)$ are convergent. Moreover, they are continuous in $t$ due to  the weak continuity of $V$ with respect to time assumed in \rf{V_space} and comments in Remark \ref{dual_space}. Thus, the system of ordinary differential equations \rf{w m} has a unique local-in-time solution $\{d_{im}(t)\}^m_{i=1}$. In view of the  {\it a priori} estimates of the sequence $\{w_m\}_{m=1}^\infty $ obtained  below in \rf{w_m_w_0_est}, the solution $d_{im}(t)$ is global-in-time. 

Multiply equation \rf{w m} by $d_{jm}$ and sum up the resulting equations for $j=1,2,..., m$. Using relation \rf{b_zero} we have $\mathfrak{b}\big( w_m , w_m , w_m\big) = \mathfrak{b}\big( V , w_m , w_m\big) =0$. Consequently,
\begin{equation*}
   \frac{1}{2}\frac{\ud}{\ud t}\|w_m(t)\|^2_2+\| \nabla w_m(t)\|_2^2 + \mathfrak{b}\big( w_m(t) , V(t), w_m (t) \big) = 0 .
\end{equation*}   
Applying inequality \rf{hardy} (or its equivalent version \rf{est_b_2}) and integrating from $s$ to $t$ yields 
\begin{equation}\label{w_m_w_0_est}
    \| w_m(t)\|^2_2 +2 \big( 1-K\sup_{t>0}\| V(t)\|_{\X} \big) \int_s^t\| \nabla w_m(\tau )\|_2^2 \ud \tau \leqslant \| w_m (s)\|_2^2 \leqslant \| w_0 \|_2^2.
\end{equation}
Recall that, by hypothesis \rf{K}, we have $K\sup_{t>0}\| V(t)\|_{\X}<1$. Thus, we can extract a subsequence, also denoted by $\{w_m\}_{m=1}^\infty $, converging towards a vector field $w \in L^2\big([0,T], \dot H_\sigma^1(\bbfR^3) \big) \cap C_w \big([0,T], L^2_\sigma (\bbfR^3) \big)$ in the following sense
\begin{align}
   \label{conv_1} & w_m \rightarrow w \quad \text{in }\ L^2\big([0,T], \dot H_\sigma^1(\bbfR^3) \big) \quad \text{weakly}\\
   \label{conv_2} & w_m \rightarrow w \quad \text{in }\ L^\infty \big([0,T], L^2_\sigma (\bbfR^3) \big)\quad \text{weak}-\star .
\end{align}

By standard arguments, involving  fractional 
 derivatives in time, see {\it e.g.} \cite[Ch. III. Thm. 3.1]{T}, it follows  that there exists a subsequence denoted again by $\{w_m\}$ such that
\begin{equation}\label{conv_3}
   w_m \rightarrow w \qquad \textrm{in} \quad L^2 \big( [0,T], L^2_{loc} (\bbfR^3 ) \big) .
\end{equation}

The next step is to show  that the limiting vector field $w = w(x,t)$ satisfies equation \rf{weak}.  By   \rf{w m}, we have for all  $\vf \in W_m$ \begin{equation}\label{w_m_plus_phi}
\begin{split} 
  \Big\langle \frac{\ud }{\ud t} w_m(t), \vf \Big\rangle  &+ \big\langle \nabla w_m (t), \nabla \vf \big\rangle + \mathfrak{b}\big( w_m(t) ,w_m(t), \vf \big) \\
&+ \mathfrak{b}\big( w_m(t),V(t) , \vf \big) + \mathfrak{b}\big( V(t) , w_m(t) , \vf  \big) =0 ,\\
   w_m (0)&= P_m w_0.
\end{split}
\end{equation}
Our goal now is to obtain  \rf{w_m_plus_phi} with a time dependent function $\vf \in C^1 \big( [0, \infty ), W_m \big)$.  The Leibniz formula yields 
\begin{equation}\label{eq_1}
  \Big\langle \frac{\ud}{\ud t} w_m (t) , \vf (t) \Big\rangle  =  \frac{\ud }{\ud t}\big\langle w_m(t), \vf (t) \big\rangle - \Big \langle w_m(t), \frac{\ud }{\ud t} \vf (t) \Big\rangle .
\end{equation}
Combining  \rf{eq_1} and \rf{w_m_plus_phi} we obtain
\begin{equation}\label{new_eq_w_m}
   \begin{split} 
    \frac{\ud }{\ud t}\big\langle w_m (t) , \vf  (t) \big\rangle & - \Big \langle w_m(t), \frac{\ud }{\ud t} \vf (t) \Big\rangle + \big\langle \nabla w_m (t), \nabla \vf (t)\big\rangle + \mathfrak{b}\big( w_m(t) ,w_m(t), \vf (t)\big) \\
&+ \mathfrak{b}\big( w_m(t),V(t) , \vf (t)\big) + \mathfrak{b}\big( V(t) , w_m(t) , \vf (t) \big) = 0 , \\
   w_m (0)&= P_m w_0.
\end{split}
\end{equation}
Integration of \rf{new_eq_w_m} over $[s,t]$ gives for all $\vf \in C^1\big( [0, \infty ), W_m\big)$
\begin{equation}\label{new_eq_w_m_1}
   \begin{split} 
    \big\langle w_m (t) , \vf  (t) \big\rangle & - \big\langle w_m (s) , \vf  (s) \big\rangle - \int_s^t \Big \langle w_m(\tau), \frac{\ud }{\ud \tau} \vf (\tau ) \Big\rangle \ud \tau \\
 & + \int_s^t \big\langle \nabla w_m (\tau ), \nabla \vf (\tau )\big\rangle \ud \tau + \int_s^t \mathfrak{b}\big( w_m(\tau) ,w_m(\tau ), \vf (\tau )\big) \ud \tau \\
&+ \int_s^t \mathfrak{b}\big( w_m(\tau ),V(\tau ) , \vf (\tau )\big) \ud \tau + \int_s^t \mathfrak{b}\big( V(\tau ) , w_m(\tau ) , \vf (\tau ) \big) \ud \tau = 0 , \\
   w_m (0)&= P_m w_0.
   \end{split}
\end{equation}
 It suffices to pass to the limit in \rf{new_eq_w_m_1} using the convergence in \rf{conv_1}--\rf{conv_3}. The first five terms in \rf{new_eq_w_m_1} are dealt as is standard when working with the classical Navier--Stokes system, (see \cite[Ch. ~III, the proof of Thm. 3.1]{T} and in particular \cite[Ch. ~III, the proof of Lemma 3.2]{T} to pass to the limit in the nonlinear term). The convergence of
\begin{align*}
  & \int_s^t \mathfrak{b}\big( w_m(\tau ),V(\tau ) , \vf (\tau ) \big) \ud \tau \rightarrow \int_s^t \mathfrak{b}\big( w(\tau ),V(\tau ) , \vf (\tau ) \big) \ud \tau  
  \intertext{and}
  & \int_s^t \mathfrak{b}\big( V(\tau ) , w_m(\tau ) , \vf (\tau ) \big) \ud \tau \rightarrow \int_s^t \mathfrak{b}\big( V(\tau ) , w(\tau ) , \vf (\tau ) \big) \ud \tau .
\end{align*}
follows by combining  Remark \ref{dual_space}, property \rf{conv_2}, and the Lebesgue dominated convergence theorem. 

The limit vector field $w = w(x,t)$ satisfies equation \rf{weak} for all $\vf \in C^1 \big( [0, \infty ), W_m \big)$ for each   $ m \geqslant 1$, and 
passing to the limit,  for all $\vf \in C \big( [0, \infty ), H^1_\sigma (\bbfR^3 ) \big) \cap C^1 \big( [0, \infty ), L^2_\sigma (\bbfR^3 ) \big)$. Hence, $w = w(x,t)$ is a weak solution of problem \rf{w:eq}--\rf{w:ini} in the energy space $X_T$ defined in \rf{XT} and, by a classical reasoning ({\it cf.} \cite[Ch. III]{T}), it satisfies strong energy inequality \rf{en:inq}.
\end{proof}

%%%%%%%%%%%%%%%%%%%%%%%%%%%%%%%%%%%%%%%%%%%%%%%%%%%%%%%%%%%%%%%%%%%%%%%%%%%%%%%%%%%%%%%%%%%%%%%%%%%%%%%%%%%%%%%%%%%%%%%%%%%%%%%%
\medskip
\section{ Generalized energy inequalities} \label{sec:GEI}

The idea behind the $L^2$-decay proof for the weak solution $w =w(x,t)$ constructed in  Theorem \ref{th:ex} is based on the  work reported in \cite{ORS}, where the decay was shown for solutions to the Navier--Stokes system with slowly decaying external forces. The extra terms in equation \rf{w:eq} containing the solution $V = V(x,t)$ cause  difficulties which do not appear in \cite{ORS}. To handle these terms, it is necessary to obtain a class of generalized energy inequalities (see \rf{gen_energy} below). The proof of such inequalities requires  stronger convergence of the approximations $\{w_m\}$ than the one stated in \rf{conv_3}. The following improvement of \rf{conv_3} seems to be well-known, however,
for the completeness, we recall the proof.
 
\begin{lemma}\label{conv_4}
There exists a subsequence of the Galerkin approximations $\{ w_m \}$ 
considered in the proof of Theorem \ref{th:ex}, which converges towards  $w=w(x,t)$ strongly in  $L^2 \big( [0,T], L^2 (\bbfR^3 ) \big)$, for every $T>0$.
\end{lemma}

\begin{corollary}\label{L_p_conv}
The sequence of the Galerkin approximations $\{w_m\}$ from Lemma \ref{conv_4} converges strongly in $L^2 \big( [0, T], L^p_\sigma (\bbfR^3) \big)$ for every $p\in [2, 6)$ and $T > 0$.
\end{corollary}

\begin{proof}[Proof of Corollary \ref{L_p_conv}.]
This is an immediate consequence of the H\"older inequality, the Sobolev inequality $\| w\|_6 \leqslant \| \nabla w\|_2$,
Lemma \ref{conv_4},  and estimate  \rf{w_m_w_0_est}. 
\end{proof}

\begin{proof}[Proof of Lemma \ref{conv_4}]
Let $\{w_m\}$ be a sequence of the Galerkin  approximations which converges towards a weak solution $w = w(x,t)$ of problem \rf{w:eq}--\rf{w:ini} in the local sense \rf{conv_3}. For every $R>0$, define the cut-off function
$\vf_R(x) = \vf_1(x/R)$, where 
\begin{equation*}
  \vf_1\in C^\infty (\bbfR^3 ) \quad \mbox{and}\quad  \vf_1(x) =
\begin{cases}1& \mbox{for}\quad  |x| \geqslant 1, \\  
 0& \mbox{for} \quad |x| \leqslant 1/2.\end{cases}
\end{equation*}
Substitute the test function $\vf  (t) = w_m (t)\vf_R^2$ into equation \rf{new_eq_w_m}. Since the function $\vf_R$ does not depend on $t$, it follows that
\begin{equation*}
    \frac{\ud }{\ud t} \langle w_m (t), w_m(t)\vf_R^2 \rangle - \big\langle w_m(t) , \frac{\ud }{\ud t}( w_m (t)\vf_R^2 )\big\rangle = \frac{1}{2} \frac{\ud }{\ud t } \| w_m(t) \vf_R \|_2^2 .
\end{equation*}
Thus, equation \rf{new_eq_w_m} yields
\begin{equation}\label{eq_2}
   \begin{split}
   \frac{1}{2} \frac{\ud }{\ud t } \| w_m(t) \vf_R \|_2^2 &+ \langle \nabla w_m (t), \nabla (w_m(t)\vf_R^2 ) \rangle + \bfo (w_m (t), w_m(t) , w_m(t)\vf_R^2 )\\
&+ \bfo (w_m(t) , V(t) , w_m(t)\vf_R^2 ) + \bfo (V(t), w_m(t) , w_m(t)\vf_R^2 ) =0.
   \end{split}
\end{equation}
By elementary calculations, we have 
\begin{align*}
   \| \nabla (w_m(t) \vf_R) \|_2^2 = \| (\nabla \vf_R) w_m(t) \|_2^2 + \| \vf_R \nabla w_m(t)\|_2^2  + \langle \nabla w_m (t), w_m(t) \nabla \vf_R^2 \rangle .
\end{align*}
The second term in \rf{eq_2} can be rewritten as 
\begin{align*}
  \langle \nabla w_m (t), \nabla (w_m(t)\vf_R^2 ) \rangle =  \| \vf_R \nabla w_m (t) \|_2^2 +  \langle \nabla w_m (t),  w_m(t)  \nabla \vf_R^2 \rangle.
\end{align*}
Combining the last two equalities with  \rf{eq_2} gives
\begin{equation}\label{1*}
    \begin{split}
   \frac{1}{2} \frac{\ud }{\ud t } \| w_m(t) \vf_R \|_2^2 &+ \| \nabla (w_m (t)\vf_R ) \|_2^2 - \| w_m (t)\nabla \vf _R\|_2^2  + \bfo (w_m (t), w_m(t) , w_m(t)\vf_R^2 )\\
   \nonumber
&+ \bfo (w_m(t) , V(t) , w_m(t)\vf_R^2 ) + \bfo (V(t), w_m(t) , w_m(t)\vf_R^2 ) =0.
   \end{split}
\end{equation}
Note that $\bfo (w_m (t), w_m(t) , w_m(t)\vf_R^2 ) = 0$. Indeed, by the definition of  the form $\bfo = \bfo (\cdot , \cdot, \cdot )$, the divergence free condition of $w_m (t)$, relation \rf{b_zero}, and since  $\vf_R$ is a scalar function, we get 
\begin{equation}\label{2*}
     \begin{split}
  \bfo ( w_m , w_m  , w_m \vf_R^2 ) &=\sum_{i,j=1}^3 \int_{\bbfR^3 } w_m^i (w_m^j)_{x_i} w_m^j \vf_R^2\ud x = \sum_{i,j=1}^3 \int_{\bbfR^3 } w_m^i \vf_R^2 (w_m^j)_{x_i} w_m^j \ud x\\
&= \bfo (w_m\vf_R^2 , w_m , w_m ) =0.
     \end{split}
\end{equation}
Similarly, we have
\begin{equation}\label{3*}
\bfo (V(t), w_m(t) , w_m(t)\vf_R^2 ) = \bfo (\vf_R^2 V, w_m , w_m ) =0.
\end{equation}
Using an analogous argument combined with inequality \rf{est_b_2} yields
\begin{equation}\label{4*}
   \begin{split}
   \big| \mathfrak{b}(w_m(t) , V(t) , w_m(t)\vf_R^2 ) \big| &= \big| \mathfrak{b}(w_m(t)\vf_R , V(t) , w_m(t)\vf_R ) \big| \\
 &\leqslant K\Big(\sup_{t>0} \| V(t)\|_{\X}\Big)\| \nabla (w_m \vf_R )\|_2^2.
    \end{split}
\end{equation}
Applying relations \rf{2*}--\rf{4*} in equation \rf{1*}, gives
\begin{equation}\label{eq_3}
   \begin{split}
   \frac{1}{2} \frac{\ud }{\ud t } \| w_m (t) \vf_R \|_2^2 + \big( 1 - K \sup_{t>0} \| V(t) \|_{\X}\big)\| \nabla (w_m(t) \vf_R ) \|_2^2 \leqslant \| w_m (t)\nabla \vf _R\|_2^2 .
   \end{split}
\end{equation}
Assumption \rf{K} insures that the second term on the left-hand side of   \rf{eq_3} is nonnegative. Thus, integration of  \rf{eq_3} from 0 to $t$ combined with  the definition of the function $\vf _R$ yields 
\begin{align*}
    \| w_m(t) \vf_R \|_2^2 &\leqslant \| w_m(0) \vf_R\|_2^2 + R^{-2} \| \nabla \vf_1\|_{\infty } \int_0^t \| w_m (s) \|_2^2 \ud s \\
    & \leqslant \| w (0) \vf_R \|_2^2 + R^{-2} T \| \nabla \vf_1 \|_{\infty } \sup_{t\in [0, T] } \| w_m (t)\|_2^2 .
\end{align*}
By \rf{w_m_w_0_est}, we have $\| w_m (t) \|_2^2 \leqslant \| w_m (0)\|_2^2 \leqslant \| w_0 \|_2^2$ and consequently, we obtain the inequality
\begin{align}\label{inq_CTR}
    \| w_m(t) \vf_R \|_2^2 \leqslant \| w(0) \vf_R\|_2^2 + CT R^{-2}\| w_0 \|_2^2 .
\end{align}
 Using the Cantor diagonal argument, we find a subsequence of  the Galerkin approximations $\{ w_m \}$ constructed in Theorem \ref{th:ex} which converges towards a weak solution $w$ in $L^2\big( [0,T], L^2(B_R) \big)$ for every ball $B_R$ of radius $R>0$. Since, the tail estimates \rf{inq_CTR} are independent of $m$, this convergence holds true in the norm of the space $L^{2}\big([0,T], L^2(\bbfR^3)\big)$. This completes the proof of the lemma.
\end{proof}

We now  prove a class of  needed generalized energy inequalities for 
 weak solutions to the perturbed problem \rf{w:eq}--\rf{w:ini}. In the sequel, 
 the following notation is introduced. For a   vector field $w= w(x,t)$   and a scalar function $\psi=\psi(x,t) $,
both depending on $x$ and $t$,  set 
\[\psi \ast w = (\psi \ast w_1, \psi \ast w_2, \psi \ast w_3 ) \quad
\mbox{and}\quad \psi ' = \partial_t \psi, \]
where the convolution is calculated  with respect to the $x$-variable, only.

% In the following theorem, $\psi \in C^1\big( [0, \infty ), \mathcal{S} (\bbfR^3)\big)$ if $\psi \in C^1\big( \bbfR^3 \times [0, \infty )\big),$ 
% and $\psi (t) \in \mathcal{S}(\bbfR^3 )$ for every $t\geqslant 0$.

\begin{theorem}[Generalized energy inequality]\label{th:ex1}
Let $E \in C^1[0, \infty )$ with $E(t) \geqslant 0$, and $\psi (t) \in C^1\big([0, \infty ); \mathcal{S}(\bbfR^3)\big)$ be arbitrary functions. Then there exists a weak solution $w \in X_T= C_w \big( [0,T], L^2_\sigma (\bbfR^3)\big)\cap L^2\big([0,T], \dot H_\sigma^1(\bbfR^3)\big)$ of the perturbed problem \rf{w:eq}--\rf{w:ini} satisfying the following generalized energy inequality
\begin{equation}\label{gen_energy}
   \begin{split}
      E(t) \|\psi(t) &\ast w(t)\|_2^2 \leqslant  E(s) \|\psi(s) \ast w(s)\|_2^2 +  \int_s^t  E'(\tau)  \|\psi(\tau) \ast w(\tau)\|_2^2  \ud \tau\\
&+2 \int_s^t E(\tau)  \Big[\langle  \psi '(\tau) \ast  w(\tau), \psi(\tau) \ast  w(\tau)\rangle - \| \psi (\tau ) \ast \nabla w( \tau )\|_2^2  \Big] \ud \tau  \\
&+2 \int_s^t E(\tau) \Big[ \mathfrak{b}(w, w, \psi \ast \psi \ast w)(\tau)   +  \mathfrak{b}(V, w, \psi \ast \psi \ast w) (\tau) \Big] \ud \tau \\
&+ 2 \int_s^t E(\tau) \mathfrak{b}(w, V, \psi \ast \psi \ast w)(\tau) \ud \tau 
    \end{split}
\end{equation} 
for almost all $s\geqslant  0$ including $s = 0$ and all $t \geqslant s \geqslant 0$.
\end{theorem}

Before proving Theorem \ref{th:ex1}, we note that the proof of the $L^2$-decay of $w = w(x,t)$ requires the following two corollaries, which are consequences of the generalized energy inequality \rf{gen_energy}.

\begin{corollary}\label{c.low} 
Let $w $ be a weak solution to \rf{w:eq}--\rf{w:ini} satisfying the generalized energy inequality \rf{gen_energy}. Then for every $\vf \in \mathcal{S}(\bbfR^3 )$, we have
\begin{align*}
     \| \vf \ast w\|_2^2 \leqslant &\|\se^{(t-s)\Delta} \vf  \ast w(s) \|_2^2 \\
+ 2 \int_s^t  &\mathfrak{b}\big( w , w , \se^{2(t-\tau )\Delta}(\vf \ast \vf ) w\big) (\tau ) + \mathfrak{b}\big( V , w , \se^{2(t-\tau )\Delta}(\vf \ast \vf ) w\big) (\tau ) \\
&+ \mathfrak{b}\big( w , V , \se^{2(t-\tau )\Delta}(\vf \ast \vf ) w\big) (\tau )  \ud \tau 
\end{align*}
for almost all $s\geqslant 0$, including $s=0$ and all $t\geqslant s$.
\end{corollary}

\begin{proof}
Use the generalized energy inequality \rf{gen_energy} with 
\begin{equation*}
    E(t) \equiv 1 \ \ \mbox{and }\ \  \psi (\tau ) = \se^{(t+\eta-\tau)\Delta } \vf,\,\mbox{where}\  \eta>0.
\end{equation*} 
Then, $\psi(\tau) \ast w(\tau)  = \se^{(t+\eta-\tau )\Delta } \vf \ast w(\tau )$ and $\psi(\tau) \ast \psi (\tau)  = \se^{2(t+\eta-\tau )\Delta } \vf \ast \vf $.  A straightforward calculation involving properties of solutions to the heat equation shows 
\begin{equation*}
          \langle \psi'(\tau) \ast  w(\tau), \psi (\tau) \ast  w(\tau)\rangle - \| \nabla\psi(t) \ast w(t)\|_2^2 =0.
\end{equation*}
The corollary follows by letting $\eta \to 0$. 
\end{proof}

\begin{corollary}\label{c.high}
Let $E \in C^1[0, \infty )$ and $E(t) \geqslant 0$. Let $w = w(x,t)$ be a weak solutions  constructed in  Theorem \ref{th:ex} satisfying the generalized energy inequality \rf{gen_energy}. Then for every $\vf \in \mathcal{S} (\bbfR^3)$, the vector field $w = w(x,t)$ fulfills the inequality 
\begin{equation}\label{ineq_high}
   \begin{split}
   E(t) \| w(t) &- \vf \ast w(t) \|_2^2 \leqslant E(s) \| w(s) - \vf \ast w(s) \|^2_2 \\
   & + \int_s^t E'(\tau ) \| w(\tau ) - \vf \ast w(\tau ) \|_2^2 \ud \tau \\
   & -  2 \int_s^t E(t) \| \nabla w (\tau ) - \vf \ast \nabla w(\tau ) \|_2^2 \ud \tau \\
   & + 2 \int_s^t E(\tau)  \mathfrak{b}(w, w, \vf  \ast \vf  \ast w - 2 \vf \ast w )(\tau )   \ud \tau \\
   & + 2 \int_s^t E(\tau ) \mathfrak{b}(V, w, \vf  \ast \vf  \ast w - 2 \vf \ast w )(\tau )   \ud \tau \\
   & + 2 \int_s^t E(\tau) \mathfrak{b}(w, V, w - 2 \vf \ast w  + \vf  \ast \vf  \ast w)(\tau)   \ud \tau 
   \end{split}
\end{equation}
for almost all $s\geqslant 0$, including $s=0$ and all $t\geqslant s$.
\end{corollary}

\begin{proof}
Substitute in the generalized energy inequality \rf{gen_energy} the function $\psi (x,t) = \zeta _n (x) - \vf (x)$, where $\zeta _n (x) = n^{-3} \zeta (x/n)$ is a smooth and compactly supported approximation of the Dirac measure. The term in \rf{gen_energy} containing $\psi '$ is annihilated. Notice also that 
\begin{equation*}
    \psi \ast \psi \ast w = \zeta _n \ast \zeta _n \ast w - 2 \zeta _n \ast \vf \ast w + \vf \ast \vf \ast w .
\end{equation*}
Since $\mathfrak{b}(w, w, w) = \mathfrak{b}(V, w, w) =0$ for a divergence free vector field $w$ ({\it cf.} Remark \ref{b_def}), passing to the limit $n \rightarrow \infty $, yields  inequality \rf{ineq_high}. 
\end{proof}

\begin{proof}[Proof of Theorem \ref{th:ex1}]
 Let $\{w_m \}$ be a sequence of Galerkin approximations converging towards a weak solution $w$ of the perturbed problem \rf{w:eq}--\rf{w:ini} in the usual sense \rf{conv_1}--\rf{conv_2} as well as in the global $L^2$-sense established in Lemma \ref{conv_4}. In the sequel  we also assume that $w_m$ is a finite linear combination of elements of an orthonormal basis $\{g_m\}$ of $L^2_\sigma (\bbfR^3)$ satisfying the properties stated in the following lemma.

\begin{lemma}\label{est_proj}
There exists an orthonormal basis $\{g_m\}_{m=1}^\infty$ of $L^2_\sigma (\R^3)$ such that 
\begin{itemize}
\item $\{g_m\}_{m=1}^\infty$ is a Riesz basis of the Sobolev space $W^{1,p}_\sigma(\R^3)$ for each $1<p<\infty$;

\item there exists $C=C(p)>0$ such that for every $v\in W^{1,p}_\sigma(\R^3)$ and every $m\in \mathbb{N}$ we have 
\begin{equation}\label{PmSob}
     \|P_mv\|_{W^{1,p}_\sigma}\leq C\|v\|_{W^{1,p}_\sigma},
\end{equation}
where $P_m v=\sum_{k=1}^m \langle v,g_k\rangle g_k$ is the
 orthonormal $L^2$-projection,

\item for every $v\in H^{1}_\sigma(\R^3)$
\begin{equation}\label{sum_nabla}
       \sum_{k=1}^\infty \langle \nabla v,g_k\rangle g_k =\nabla v = \sum_{k=1}^\infty \langle  v,g_k\rangle \nabla g_k ,
\end{equation}
where the series converges strongly in $L^2_\sigma(\R^3)$.
\end{itemize}
\end{lemma}

\begin{proof}
These properties are satisfied by the divergence free vector wavelet basis introduced by Battle and Federbush \cite{BF93}. See also \cite[Ch.~3]{L12} for a review of properties of such  bases. These wavelets decay exponentially. A standard application of the  Calder\'on-Zygmund theory yields  inequality \rf{PmSob} using the size and moments estimates for these wavelets.
\end{proof}

\begin{corollary}\label{C_conv_6}
Under the assumptions of Lemma \ref{est_proj}, for every $v\in H^1_\sigma (\bbfR^3) $ we have 
\begin{equation}\label{conv_6}
    \| P_m v - v \|_{H^1_\sigma } \rightarrow 0 \qquad \textit{as}\quad m \rightarrow \infty .
\end{equation}
\end{corollary}

\begin{proof}
Immediate properties of the $L^2$-projection $P_m$  give $\lim_{m\rightarrow \infty } \| P_m v - v\|_2 = 0$. Note now that we have $\nabla P_m v = \sum_{k=1}^{m} \langle v , g_k \rangle \nabla g_k. $
\ Thus,  by the second equality in  \rf{sum_nabla} we obtain $\lim_{m\rightarrow \infty }\| \nabla P_m v - \nabla v \|_2 = 0$.
\end{proof}

We now return to the proof of Theorem \ref{th:ex1}. Use \rf{new_eq_w_m_1}
with  $\vf (t)$ replaced by    the following test function
\begin{equation}\label{phi:def}
   E(t) \vf_m (t) = E(t) P_m \big( w_m (t) \ast \psi (t)  \ast  \psi (t)\big)  ,
\end{equation}
where  $E \in C^1\big([0, \infty )\big)$, $E(t) \geqslant 0$ and $\psi (t) \in C^1\big( [0, \infty ); \mathcal{S}(\bbfR^3)\big)$. Here, $P_m : L^2_\sigma (\bbfR^3) \rightarrow W_m$ is the usual orthogonal projection (see the proof of the first part of Theorem \ref{th:ex}), thus, $\vf_m = P_m (w_m \ast \psi \ast \psi ) \in C^1 \big( [0, \infty ) , W_m \big)$.

Using properties of the projection $P_m$ yields
\begin{equation}\label{proj}
    \langle v_m , P_m (v) \rangle = \langle v_m, v \rangle \quad \textrm{for all} \quad v_m \in W_m \quad \textrm{and} \quad  v \in L^2_\sigma (\bbfR^3 ).
\end{equation}
By the Leibniz formula, we  have 
\begin{equation}\label{ddtEP}
   \begin{split}
      \frac{\ud }{\ud t} \Big( E P_m\big(w_m \ast \psi  \ast \psi \big) \Big) = &E' P_m\big( w_m \ast \psi  \ast \psi \big)
 + 2E P_m\big( w_m  \ast \psi '  \ast \psi \big)\\
&+ E P_m\bigg( \Big(\frac{\ud }{\ud t}w_m \Big)\ast \psi  \ast \psi \bigg)\ .
   \end{split}
\end{equation}
A combination of   \rf{proj}, \rf{phi:def},  and  properties of the convolution gives 
\begin{equation}\label{wEP}
   \begin{split}
      \bigg\langle w_m ,E P_m\Big( \Big(\frac{\ud }{\ud t}w_m \Big)\ast \psi  \ast \psi  \Big) \bigg\rangle 
 &= \bigg\langle w_m ,E  \Big(\frac{\ud }{\ud t}w_m \Big)\ast \psi  \ast \psi  \bigg\rangle \\
& = \Big\langle \frac{\ud }{\ud t} w_m , E w_m \ast \psi \ast \psi \Big \rangle 
 = \Big\langle \frac{\ud }{\ud t} w_m , \vf_m \Big \rangle.
   \end{split}
\end{equation}
Equality \rf{w_m_plus_phi} can be rewriten as 
\begin{equation}\label{wmphim}
   \begin{split}
      \Big\langle \frac{\ud }{\ud t} w_m , \vf_m \Big \rangle =& - \langle \nabla w_m , \nabla \vf_m \rangle - \mathfrak{b}( w_m, w_m, \vf_m ) \\
& - \mathfrak{b}( w_m, V, \vf_m ) -  \mathfrak{b}( V, w_m, \vf_m ).  
   \end{split}
\end{equation}
After substituting 
 $\vf  = E\vf_m = EP_m (w_m \ast \psi \ast \psi )$ into \rf{new_eq_w_m_1} and using \rf{proj}--\rf{wEP}, we obtain the equality  
\begin{equation}\label{int_s_t}
  \begin{split}
    E(t) \langle w_m (t) , \vf_m (t) \rangle  = &E(s) \langle w_m (s) , \vf_m (s) \rangle 
     + \int_s^t E' (\tau ) \big\langle w_m (\tau ) , \vf_m (\tau ) \big\rangle \ud \tau \\
    & + 2 \int_s^t E( \tau ) \big\langle w_m ( \tau ), w_m (\tau ) \ast \psi ' (\tau )  \ast  \psi (\tau )  \big\rangle \ud \tau \\
    & - 2 \int_s^t E(\tau) \langle \nabla w_m (\tau ) , \nabla \vf_m (\tau ) \rangle \ud \tau \\
    & - 2 \int_s^t  E(\tau) \mathfrak{b} \big( w_m (\tau ) , w_m (\tau ) , \vf_m (\tau ) \big) \ud \tau \\
    & -2 \int_s^t  E(\tau) \mathfrak{b} \big( w_m (\tau ) , V (\tau ) , \vf_m (\tau ) \big) \ud \tau \\
    & - 2 \int_s^t  E(\tau) \mathfrak{b} \big( V(\tau ) , w_m (\tau ) , \vf_m (\tau ) \big) \ud \tau .
  \end{split}
\end{equation}

Passing to the limit  as $m \rightarrow \infty $ in \rf{int_s_t}  
will yield the generalized inequality \rf{gen_energy}. 
Here, we  adapt arguments from \cite[Proof of Prop. 2.3]{ORS} to our more general case. 
We need the following result on the convergence of the sequence $\vf_m$.

\begin{lemma}\label{lem_conv_7}
Under the above assumptions 
\begin{equation}\label{conv_7}
   \vf_m = P_m( w_m \ast \psi \ast \psi ) \rightarrow w\ast \psi \ast \psi \qquad \textit{as}\quad m \rightarrow \infty  
\end{equation}
in $L^2\big( [0, T], H^1_\sigma (\bbfR^3) \big)$ for each $T>0$. Moreover, for each $p\in [2, \infty )$ there exists a constant $C = C(p, \psi ) >0$ such that
\begin{equation}\label{est_conv_7}
    \| \nabla \vf_m (t) \|_p \leqslant C \| w_0 \|_2
\qquad \text{for all}\quad t>0.
\end{equation}
\end{lemma}

\begin{proof}
By Lemma \ref{conv_4}, 
 $w_m (t) \rightarrow w(t)$ strongly in $L^2_\sigma (\bbfR^3)$ for almost all $t > 0$, hence, by properties of the convolution, we have
\begin{equation*}
   \nabla \big( w_m \ast \psi \ast \psi \big)(t) =
\big( w_m \ast \nabla \psi \ast \psi \big)(t)
\rightarrow \big(w\ast \nabla \psi \ast \psi \big)(t)
=\nabla \big( w \ast \psi \ast \psi \big)(t) 
\end{equation*}
strongly in $L^2_\sigma (\bbfR^3)$ for almost all $t > 0$, as well. Thus, by Corollary \ref{C_conv_6}, we obtain 
\begin{equation}\label{Pm_conv}
    P_m \big( w_m \ast \psi \ast \psi \big) (t) \rightarrow \big( w\ast \psi \ast \psi \big) (t) 
\end{equation}
strongly in $H^1_\sigma (\bbfR^3)$ for almost all $t > 0$. Finally, by \rf{PmSob} and by \rf{w_m_w_0_est}, we obtain 
\begin{align*}
    \| P_m\big( w_m \ast \psi \ast \psi \big) (t) \|_{H^1_\sigma } &\leqslant C\| w_m \ast \psi \ast \psi (t)\|_{H^1_\sigma } \\
    & \leqslant C\| w_m (t) \|_2 \Big( \| \psi \ast \psi (t) \|_1 + \| \big( \nabla \psi \ast \psi \big) (t)\|_1 \Big) \\
    & \leqslant C(\psi ) \| w_0\|_2 .
\end{align*}
Hence, we may apply the Lebesgue dominated convergence theorem to complete the proof of the first part of Lemma \ref{lem_conv_7}.

To prove the second part, we use estimate \rf{PmSob} and \rf{w_m_w_0_est} as well as the Young inequality for  convolutions with exponents $p$ and $q$ satisfying $1/p + 1 = 1/2 + 1/q$ to obtain
\begin{align*}
   \| \nabla \vf_m (t) \|_p &\leqslant C \| \big( w_m \ast \psi \ast \psi \big) (t)\|_{W^{1,p}} 
\leqslant C \| w_m(t) \|_2 \Big( \| \psi \ast \psi \|_q + \| \nabla \psi \ast \psi \|_q \Big) \leqslant C \| w_0\|_2. 
\end{align*}
\end{proof}

Now, we are in a position to pass to the limit in each term of \rf{int_s_t}.

Since $w_m (t) \rightarrow w(t) $ strongly in $L^2_\sigma (\bbfR^3)$ for almost all $t > 0$, we also have
\begin{equation*}
    \vf_m (t) = P_m \big( w_m \ast \psi \ast \psi \big) (t) \rightarrow \big( w \ast \psi \ast \psi \big) (t) 
\end{equation*}
strongly in $L^2_\sigma (\bbfR^3)$.  Consequently,
\begin{align*}
    E(t) \langle w_m(t), \vf_m (t) \rangle \rightarrow E(t) \langle w(t), w(t) \ast \psi (t) \ast \psi (t) \rangle  = E(t) \| w (t) \ast \psi (t) \|_2^2  
\end{align*}
almost everywhere in $t$.

The convergence of the sequence $\{w_m\}$ in $L^2\big( [0, T], L^2_\sigma (\bbfR^3 ) \big)$ and Lemma \ref{lem_conv_7} allows us to pass to the limit in the second and third term on the right-hand side of \rf{int_s_t} and to obtain, as $m\rightarrow \infty $, 
\begin{align*}
    \int_s^t E' (\tau ) \langle w_m (\tau ) , \vf_m (\tau ) \rangle \ud \tau \rightarrow & \int_s^t E' (\tau ) \langle w (\tau ) , w (\tau ) \ast \psi (\tau ) \ast \psi (\tau ) \rangle \ud \tau \\
    &= \int_s^t E' (\tau ) \| w(\tau ) \ast \psi (\tau )\|^2_2 \ud \tau
\end{align*}
and 
\begin{equation*}
    \int_s^t E (\tau ) \langle w_m (\tau ) , w_m (\tau ) \ast \psi' (\tau ) \ast \psi (\tau ) \rangle \ud \tau \rightarrow \int_s^t E (\tau ) \langle w (\tau ) , w (\tau ) \ast \psi '(\tau ) \ast \psi (\tau ) \rangle \ud \tau .
\end{equation*}
It follows from the weak convergence \rf{conv_1} that 
\begin{align*}
   \liminf_{m\rightarrow \infty } \int_s^t E(\tau ) \langle \nabla w_m (\tau ), \nabla \vf_m (\tau ) \rangle \ud \tau &\geqslant  \int_s^t E(\tau ) \langle \nabla w (\tau ), \nabla \big( w(\tau ) \ast \psi (\tau ) \ast \psi (\tau )\big) \rangle \ud \tau \\
   &= \int_s^t E(\tau ) \| \nabla w( \tau ) \ast \psi (\tau ) \|_2^2 \ud \tau .
\end{align*}

Using properties of the trilinear form $\mathfrak{b}(\cdot , \cdot , \cdot )$ recalled in Remark \ref{form_b} the nonlinear term (the fifth one on the right-hand side of \rf{int_s_t} ) is estimated  as follows
\begin{align*}
   \Big| \int_s^t E(\tau ) \mathfrak{b}(w_m , w_m  ,  \vf_m ) (\tau )& \ud \tau - \int_s^t E(t)\mathfrak{b}( w, w ,  w \ast \psi   \ast  \psi )(\tau ) \ud \tau \Big| \\
    \leqslant &\Big| \int_s^t E(\tau ) \Big\langle \big( w_m  -w  \big) \cdot \nabla \vf_m   , w_m \Big\rangle (\tau ) \Big| \ud \tau\\
   & + \Big|  \int_s^t E(\tau )\Big\langle w  \cdot \nabla \vf_m  , w_m - w   \Big\rangle (\tau ) \Big| \ud \tau\\
   & + \Big|  \int_s^t E(\tau )\Big\langle w  \cdot \nabla \big( \vf_m - w \ast \psi  \ast \psi  \big) , w \Big\rangle (\tau ) \Big| \ud \tau\\
   \equiv& I_1 + I_2 + I_3.
\end{align*}
To estimate the term $I_1$,  we use the H\"older inequality, estimate \rf{est_conv_7} with $p = 4$, inequality \rf{w_m_w_0_est}, 
and Corollary \ref{L_p_conv}
to obtain
\begin{align*}
    I_1 &\leqslant \int_s^t E(\tau )\| w_m (\tau ) - w(\tau )\|_4 \| \nabla \vf_m (\tau ) \|_4 \| w_m (\tau )\|_2 \ud \tau \\
    & \leqslant C \| E \|_{L^\infty (0, \infty )} \| w_0 \|_2^2 \int_s^t \| w_m (\tau ) - w(\tau )\|_4  \ud \tau \\
    & \leqslant C(E, w_0) (t-s)^{1/2} \int_s^t  \| w_m (\tau ) - w(\tau )\|_4^2   \ud \tau \rightarrow 0 \quad \textrm{as} \quad m\rightarrow \infty. 
\end{align*}
 An analogous reasoning applies  to the term $I_2$.

Estimates for $I_3$ are similar. By the H\"older inequality and  the well-known inequality
$\| w\|_4 \leqslant C \| w\|_2^{1/4} \| \nabla w\|_2^{3/4}$ we have
\begin{align*}
    I_3 \leqslant &\| E \|_{L^\infty (0, \infty )} \int_s^t | \| w(\tau )\|_4^2  \| \nabla \big( \vf_m - w\ast \psi \ast \psi \big) (\tau )\|_2 \ud \tau \\
     \leqslant &C(E) \int_s^t \| w(\tau )\|_2^{1/2}  \| \nabla w(\tau )\|_2^{3/2} \| \nabla \big( \vf_m  - w \ast \psi \ast \psi \big) (\tau )\|_2\ud \tau \\
     \leqslant &C(E ) \ve \int_s^t \| w (\tau )\|_2^{2/3} \| \nabla w(\tau ) \|_2^2 \ud \tau \\
    & + C(E, \ve ) \int_s^t \| \nabla \big( \vf_m - w\ast \psi \ast \psi \big) (t) \|_2^4 \ud \tau .
\end{align*}
The first term is arbitrarily small with $\ve > 0$ because by the energy inequality \rf{en:inq}, the quantity $\int_s^t \| w (\tau )\|_2^{2/3} \| \nabla w(\tau ) \|_2^2 \ud \tau $ may be bounded by a multiple of $\|w_0\|_2^{2+2/3}$. The second term converges to zero as $m \rightarrow \infty $ which results from Lemma \ref{lem_conv_7} and from the estimate
\begin{equation*}
   \| \nabla \big( \vf_m - w\ast \psi \ast \psi \big) (t) \|_2 \leqslant \| \nabla \vf_m \|_2 + \| w \ast \nabla \psi \ast \psi \|_2 \leqslant C\| w_0\|_2
\end{equation*}
being a direct consequence of estimate \rf{est_conv_7}.

The final step is to deal with the last two terms on the right-hand-side of \rf{int_s_t} which contain the solution $V=V(x,t)$. Using  the standing assumptions \rf{est_b_1} and \rf{est_b_2} yields 
\begin{align*}
    \int_s^t E(\tau )\Big| \bfo \big( w_m  , &\vf_m , V \big) (\tau ) - \bfo \big( w , w  \ast \psi \ast \psi , V \big) (\tau )\Big| \ud \tau \\
   \leqslant & \| E \|_{L^\infty (0, \infty )} K\sup_{t>0} \| V(t) ) \|_{\X} \int_s^t \| \nabla w_m (\tau ) \|_2 \| \nabla \big( \vf_m - w \ast \psi  \ast \psi \big) (\tau )\|_2 \\
   &+  \int_s^t E(\tau )\Big| \bfo \big( ( w_m - w  , w  \ast \psi  \ast \psi , V \big) (\tau )\Big| \ud \tau \\
   \leqslant& C \ve \int_s^t \| \nabla w_m (\tau )\|_2^2 \ud \tau + C(\ve ) \int_s^t \| w_m(\tau ) - w(\tau )\|_2^2 \ud \tau \\
   &+  \| E \|_{L^\infty (0, \infty )} \int_s^t \Big| \bfo \big( ( w_m  - w , w  \ast \psi  \ast \psi , V \big) (\tau )\Big| \ud \tau .
\end{align*}
The first term on the right-hand-side can be made arbitrarily small since $\ve >0$ is arbitrary and  since $\int_s^t \| \nabla w_m(s)\|_2^2 \ud s \leqslant \|w_0\|_2^2 $, (see \rf{w_m_w_0_est}). The second term converges to zero by Lemma ~\ref{lem_conv_7} and the third one converges to zero by Remark \ref{dual_space} combined with the Lebesgue dominated convergence theorem.

Analogously, we show that 
\begin{equation*}
  \int_s^t E(\tau ) \mathfrak{b} ( V, w_m, \vf ) (\tau ) \ud \tau \rightarrow \int_s^t E(\tau ) \mathfrak{b} ( V, w, w \ast \psi \ast \psi ) (\tau ) \ud \tau \qquad \textrm{as} \quad m\rightarrow \infty .
\end{equation*}
This completes the proof of the generalized energy inequality \rf{gen_energy}.
\end{proof}
%%%%%%%%%%%%%%%%%%%%%%%%%%%%%%%%%%%%%%%%%%%%%%%%%%%%%%%%%%%%%%%%%%%%%%%%%%%%%%%%%%%%%%%%%%%%%%%%%%%%%%%%%%%%%%%%%%%%%%%%%%%%%%%%%%%%%%%%%%%%%%%%%%%%%%%%%%%%%%%%%%%%%%%%%%%%%%%%%%%%%%%%%%%%%%%%
\section{Asymptotic stability of weak solutions}\label{splitting}
\setcounter{equation}{0}

The proof of the $L^2$-decay of a solution to the perturbed problem \rf{w:eq}--\rf{w:ini} is somewhat challenging and we use {\it the  Fourier splitting}, a technique that  was  introduced  in \cite{S1,S} and generalized in \cite{ORS}. 

\begin{proof}[Proof of Theorem \ref{th:ex}. Second part -- decay of  solutions.]
The $L^2$-norm of the solution $w = w(x,t)$ from Theorem \ref{th:ex1} is decomposed into high and low frequency terms. To deal with these two  terms,  the generalized energy inequality \rf{gen_energy}  is used with suitable functions $E(t)$ and $\psi $, as were used in Corollaries \ref{c.low} and \ref{c.high}. A modification of the Fourier splitting argument will be applied to estimate the term corresponding to high frequencies. 

We begin by decomposing the $L^2$-norm of the Fourier transform of $w$ as follows
\begin{equation} \label{low-high} 
        \| w(t) \|_2 = \| \widehat{w}(t)\|_2  \leqslant  \| \check{\vf } \widehat{w}(t)\|_2 +  \| (1-\check{\vf} ) \widehat{w}(t)\|_2, \qquad \textrm{where}\quad  \check{\vf }(\xi) = e^{-|\xi|^2}. 
\end{equation}
Each term on the right hand side of the last inequality will be estimated separately. Notice that $\check{\vf }$ is the inverse Fourier transform of the function $\vf (x) = (4\pi )^{-3/2} \se^{-|x|^2/4}$, which  is the fundamental solution of the heat equation at $t = 1$.

{\it Estimates of  the low frequencies.}
Using the Plancherel identity and Corollary \ref{c.low}, we have 
\begin{equation}
    \begin{split}
    (2\pi )^{\frac{3}{2}} \| \check{\vf }\widehat{w}(t)\|_2^2 &=  \| \vf  \ast w(t) \|_2 \\
&\leqslant \| \se^{(t-s)\Delta }\vf \ast w(s)\|_2^2 \\
&+  2(2 \pi )^{-\frac{3}{2}} \int_s^t \big| \mathfrak{b} (w, w, \se^{2(t-\tau )\Delta}\vf \ast \vf \ast w)  (\tau ) \big| \ud \tau \\
&+  2(2 \pi )^{-\frac{3}{2}} \int_s^t \big| \mathfrak{b} (V, w,\se^{2(t-\tau )\Delta} \vf  \ast \vf\ast w ) (\tau ) \big| \ud \tau \\
&+2(2 \pi )^{-\frac{3}{2}} \int_s^t \big| \mathfrak{b}( w , V, \se^{2(t-\tau )\Delta} \vf  \ast  \vf  \ast  w ) (\tau )\big| \ud \tau  \\ 
&=I_1(t,s) + 2(2 \pi )^{-\frac{3}{2}}\int_s^t \big( I_2 (\tau ) + I_3 (\tau ) + I_4(\tau )\big) \ud \tau .
    \end{split}
\end{equation}
Quantities $I_i$ are going to be estimated  separately for each $i=1, 2, 3, 4$.

Note that $ \check{\vf  }\widehat{w}(s) \in L^2( \bbfR^3)$. Hence, for each fixed $ s\geqslant 0$,  by the Plancherel identity  and  the Lebesgue dominated convergence theorem it follows that
\begin{equation}\label{I}
      I_1(t,s) =  \| \se^{(t-s)\Delta }\vf \ast w(s)\|_2^2 = (2\pi )^3 \| \se ^{-(t-s) |\xi |^2} \check{\vf }\widehat{w}(s)\|_2^2 \rightarrow 0 \qquad \textrm{as }\ t\rightarrow \infty .
\end{equation}

Applying the Schwarz inequality, the well-known $L^2$-estimate for the heat semigroup, the H\"older inequality, the Sobolev inequality and, finally, the energy inequality \rf{en:inq}, we obtain 
\begin{equation} \label{I_2}
   \begin{split}
     I_2 &=  \big| \mathfrak{b}( \vf \ast \vf \ast w, w,\se^{2(t-\tau )\Delta} w ) \big|\\
&\leqslant \| \vf \ast \vf \ast (w\cdot \nabla w)\|_2 \| \se^{2(t-\tau )\Delta }w\|_2 \leqslant C \| \vf \ast \vf \|_{\frac{6}{5}}\|w\cdot \nabla w\|_{\frac{3}{2}} \|w\|_2 \\
&\leqslant  C \| \vf \ast \vf \|_{\frac{6}{5}} \|w\|_6 \|\nabla w\|_2 \|w\|_2 \leqslant  C \| \vf \ast \vf \|_{\frac{6}{5}} \| w_0 \|_2 \|\nabla w\|_2^2 .
   \end{split}
\end{equation}

Using the standing assumption \rf{hardy} (reformulated in Remark \ref{form_b}), properties of the heat semigroup and of the convolution, we get
\begin{equation} \label{I_3}
    \begin{split}
   I_3 &=   \big|\mathfrak{b}(V, w, \se^{2(t-\tau )\Delta} \vf  \ast \vf \ast w) \big| \\
&\leqslant K\|V \|_{\X} \| \nabla w\|_2 \| \nabla \se^{2(t-\tau )\Delta  }\vf \ast \vf \ast w\|_2 \\
&\leqslant  K \| V \|_{\X}\| \nabla w\|_2 \| \vf \ast \vf \ast \nabla w\|_2 \leqslant  K \| \vf \ast \vf \|_1  \|\nabla w\|_2^2 \sup_{t>0} \| V(t)\|_{\X},
    \end{split}
\end{equation}
where $K>0$ is the constant from inequality \rf{hardy}. By an analogous argument involving assumption \rf{hardy}, we have
\begin{equation} \label{I_4}
    \begin{split}
   I_4 &=  \big| \mathfrak{b}( w, V,\se^{2(t-\tau )\Delta} w \ast \vf \ast \vf  )\big| \\ 
&= \Big| \mathfrak{b}\big( w, \se^{2(t-\tau )\Delta} w \ast \vf \ast \vf , V\big) \Big| \leqslant K \| \vf \ast \vf \|_1 \|\nabla w\|_2^2  \sup_{t>0} \| V(t)\|_{\X}.
    \end{split}
\end{equation}

Combining estimates \rf{I_2}--\rf{I_4}, we obtain the following  $L^2$-bound of  the low frequencies
\begin{equation} \label{low}
     \|\check{\vf }\widehat{w}(t)\|_2^2 \leqslant  I_1(t,s) + C \Big( \| w_0 \|_2 + \sup_{t >0} \| V(t) \|_{\X}\Big)\int_s^t  \|\nabla w(\tau )\|_2^2 \ud \tau ,
\end{equation}
where the constant $C=C(K, \check{\vf} ) >0$ is independent of $w$.

Let $s>0$ be fixed and large on the right-hand side of inequality \rf{low}. The term $I_1(t,s)$ tends to zero as $t\rightarrow \infty $ by \rf{I}. Since $\int_0^\infty \| \nabla w( \tau )\|_2^2 \ud \tau <\infty $, the quantity $\int_s^\infty  \|\nabla w(\tau )\|_2^2 \ud \tau $ can be made arbitrary small choosing $s$ large enough. Hence, $\|\check{\vf }\widehat{w}(t)\|_2^2 \rightarrow 0$ as $t\rightarrow \infty $.

{\it Estimates of the high frequencies.}
To deal with the term $\| (1-\check{\vf} ) \widehat{w}(t)\|_2$ in  the decomposition \rf{low-high}, we use Corollary \ref{c.high} with the test function $\vf $ satisfying $\check{\vf }(\xi ) =\exp{(-|\xi|^2)}$ and  a function $E(t) >0$ to be determined below.  We  then apply the Fourier-splitting method to estimate each term on the right-hand side of \rf{ineq_high}.

For every $G (t) \geqslant 0$,  the Plancherel formula 
applied to the second and third term 
on the right hand side of  \rf{ineq_high} yields 
\begin{align*}
         J_1& = \int_s^t E'(\tau ) \| w(\tau ) - \vf \ast w(\tau )\|^2_2 \ud \tau - 2\int_s^t E(\tau ) \|\nabla w(\tau ) - \vf \ast \nabla w(\tau )\|_2^2 \ud \tau \\
& = \int_s^t  E'(\tau)  \int_{|\xi | >G(t)}\Big| \big( 1-\check{\vf }(\xi)\big) \widehat{w}(\xi , \tau) \Big|^2 \ud \xi  \ud \tau  \\
&-2\int_s^t E(\tau)  \int_{|\xi | >G(t)} \Big| |\xi| \big( 1-\check{\vf }(\xi) \big) \widehat{w}(\xi , \tau) \Big|^2 \ud \xi  \ud \tau\\
&+\int_s^t  E'(\tau)  \int_{|\xi | \leqslant G(t)}\Big| \big( 1-\check{\vf }(\xi)\big) \widehat{w}(\xi , \tau) \Big|^2 \ud \xi  \ud \tau  \\
&-2\int_s^t E(\tau)  \int_{|\xi | \leqslant G(t)} \Big| |\xi| \big( 1-\check{\vf }(\xi) \big) \widehat{w}(\xi , \tau) \Big|^2 \ud \xi  \ud \tau .
\end{align*}
Now,  choose
\begin{equation}\label{fun_G_E}
   E(t) = (1+t)^{\alpha} \qquad\textrm{and } \qquad G^2 (t) = \frac{\alpha}{2(t+1)} \qquad \textrm{with fixed}\quad \alpha >0,
\end{equation}
then  $E'(\tau )-2E(\tau ) G^2(\tau )=0$. Thus,
\begin{align*}
       \int_s^t  E'(\tau)  &\int_{|\xi | >G(t)}\Big| \big( 1-\check{\vf }(\xi)\big) \widehat{w}(\xi , \tau) \Big|^2 \ud \xi  \ud \tau  -2\int_s^t E(\tau)  \int_{|\xi | >G(t)} \Big| |\xi| \big( 1-\check{\vf }(\xi) \big) \widehat{w}(\xi , \tau) \Big|^2 \ud \xi  \ud \tau\\
&\leqslant \int_s^t \Big[ E'(\tau )-2 E(\tau) G^2(\tau) \Big]  \int_{|\xi | >G(t)}\big| (1-\check{\vf }(\xi)) \widehat{w}(\tau) \big|^2 \ud \xi  \ud \tau =0.
\end{align*}

Moreover, for small $|\xi|$, it follows that $\big|1-\check{\vf }(\xi) \big| = 1-\se^{-|\xi|^2} \leqslant |\xi|^2$. Hence, using \rf{fun_G_E}, we obtain
\begin{align*}
    \int_s^t  E'(\tau) & \int_{|\xi | \leqslant G(t)}\Big| \big( 1- \check{\vf }(\xi) \big) \widehat{w}(\xi , \tau )\Big|^2 \ud \xi  \ud \tau \\
&= \int_s^t E'(\tau ) \int_{|\xi | \leqslant G(t)}|(1-\check{\vf }(\xi)) \widehat{w}(\xi , \tau )|^2 \ud \xi  \ud \tau \\
&\leqslant  C \|w_0\|_2^2 \int_s^t E'(\tau )G^4(\tau ) \ud \tau \leqslant C \int_s^t (1+\tau )^{\alpha -3} \ud \tau .
\end{align*}
Hence, since $E(t) \geqslant 0$, we conclude that  
\begin{equation}\label{cota} 
   J_1 \leqslant C \int_s^t (1+\tau )^{\alpha -3} \ud \tau .
\end{equation}

To deal with the other terms on the right-hand side of the inequality \rf{ineq_high},  set $\eta = \vf \ast \vf - 2 \vf$ to simplify the notation. Combining  the H\"older and the Young inequalities with the Sobolev inequality $\|w \|_6 \leqslant \| \nabla w \|_2$ (as in the estimates of low frequencies in \rf {I_3}) yields
\begin{equation}\label{J2}
   \begin{split}
&  \int_s^t E(\tau) \big| \mathfrak{b}( w, w, \vf  \ast \vf  \ast w - 2 \vf \ast w )(\tau )  \big| \ud \tau \\
&  = \int_s^t E(\tau ) \big| \mathfrak{b}( w, w,\eta \ast w ) (\tau ) \big| \ud \tau \leqslant \int_s^t E(\tau ) \| w(\tau ) \cdot \nabla w(\tau ) \|_{\frac{3}{2}} \| \eta \ast w(\tau )\|_3 \ud \tau \\ 
&\leqslant \|\eta\|_{\frac{6}{5}}\int_s^t E(\tau )  \| w(\tau )\|_6 \|\nabla w(\tau )\|_2 \| w(\tau )\|_2 \ud \tau \leqslant    C  \|\eta\|_{\frac{6}{5}}\| w_0\|_2  \int_s^t E(\tau )  \| \nabla w (\tau )\|_2^2 \ud \tau .
    \end{split}
\end{equation}

The two terms on the right-hand side of \rf{ineq_high} containing the solution $V$ are estimated by the standing assumption \rf{hardy}, see Remark \ref{form_b}. Indeed, we have
\begin{equation}\label{J3}
    \begin{split}
&  \int_s^t E(\tau )\big| \mathfrak{b}(V, w,  \vf  \ast \vf  \ast w - 2 \vf \ast w )(\tau ) \big| \ud \tau \\
&= \int_s^t E(\tau ) \big| \mathfrak{b} (V, w , \eta \ast w)(\tau )  \big| \ud \tau \leqslant C \| \eta \|_1\Big( \sup_{t>0} \| V(t)\|_{\X} \Big) \int_s^t E(\tau ) \| \nabla w (\tau )\|_2^2 \ud \tau  .
     \end{split}
\end{equation}

Proceeding in an analogous way, it follows that 
\begin{equation}\label{J4}   
    \begin{split}  
   \int_s^t E(\tau)\big| \mathfrak{b}(w, V, w &- 2 \vf \ast w  + \vf  \ast \vf  \ast w)(\tau)  \big| \ud \tau  \\
&\leqslant C (1+\| \eta \|_1 )\Big( \sup_{t>0}\|V\|_{\X} \Big)  \int_s^t E(\tau )  \|\nabla w(\tau )\|_2^2  \ud \tau .
    \end{split}
\end{equation}

Thus, dividing inequality \rf{ineq_high} by $E(t)$ and using estimates \rf{cota}, \rf{J2}, \rf{J3} and \rf{J4}, we get 
\begin{equation}\label{limsup}
     \begin{split}
  \|( 1- \check{\vf }) \widehat{w}(t)\|_2^2 &\leqslant  \frac{E(s)}{E(t)} \| (1-\check{\vf }) \widehat{w}(s)\|_2^2\\ 
&+  C \frac{1}{E(t)}\int_s^t (1+\tau )^{\alpha -3} \ud \tau + C  \frac{1}{E(t)} \int_s^t E(\tau)  \|\nabla w (\tau )\|_2^2  \ud \tau .
     \end{split}
\end{equation}
For fixed $s>0$, we compute the $\limsup$  as $t\rightarrow \infty $ of both sides of \rf{limsup}. Since, $E(t) = (1 + t )^\alpha $ with some $\alpha >0$, the first term on the right-hand side of \rf{limsup} tends to zero. By a direct calculation based on the de l'H$\hat{\mbox{o}}$pital rule, we obtain 
\begin{align*}
      \limsup_{t \rightarrow \infty} \frac{1}{(t+1)^\alpha } \int_s^t (1+\tau )^{\alpha -3} \ud \tau = 0 .
\end{align*}
Thus, using the inequality $\frac{E(\tau )}{E(t)}= \big( \frac{1+\tau }{1 + t} \big)^\alpha \leqslant 1$ for $\tau \in [0,t]$, it follows from estimate \rf{limsup}
\begin{equation}\label{cosa4}
   \limsup_{t \rightarrow \infty} \| (1- \vf ) \widehat{w}(t) \|_2^2 \leqslant C \Big(\sup_{t>0} \| V(t) \|_{\X} \Big) \int_s^\infty  \| \nabla w (\tau ) \|_2^2 \ud \tau .
\end{equation}
We conclude that ${\limsup}_{t \to \infty}\|(1-\vf (\xi )) \widehat{w}(t)\|_2^2 =0$, since the right-hand side of inequality \rf{cosa4} can be made arbitrarily small for sufficiently large $s>0$. This completes the proof of Theorem \ref{th:ex}.
\end{proof}

\section*{Acknowledgments}
The authors are greatly indebted to Victor Galaktionov for drawing their attention to the Slezkin paper \cite{Slezkin} and to its translation in \cite{Gal}, to Pierre-Gilles Lemari\'e-Rieusset for helpful comments on divergence free vector wavelet bases discussed in Lemma \ref{est_proj}, and to Changxing Miao for important remarks on Theorem \ref{thm:hardy}.

This work was partially supported by  the Foundation for Polish Science operated within the Innovative Economy Operational Programme 2007-2013 funded by European Regional Development Fund (Ph.D. Programme: Mathematical
Methods in Natural Sciences).  D.~Pilarczyk was supported the MNiSzW grant No.~N~N201 418839
and G. Karch by grant  IDEAS PLUS No. IdP2011/000661.  M.  Schonbek was partially supported by the NSF Grant DMS-0900909.

\end{document}